\documentclass[11pt, reqno]{amsart}

\usepackage{amsmath,amssymb,amsthm,epsfig}
\usepackage{hyperref}
\usepackage{amsmath}
\usepackage{amssymb}
\usepackage{color}
\usepackage{ulem}
\usepackage{epsfig}
\usepackage[mathscr]{eucal}

\usepackage[utf8]{inputenc}

\usepackage{stackrel}

\usepackage{epsfig}

\newtheorem{theorem}{Theorem}
\newtheorem{definition}{Definition}

\newtheorem{lemma}{Lemma}
\newtheorem{proposition}{Proposition}
\newtheorem{corollary}{Corollary}
\newtheorem{remark}{Remark}
\newtheorem{example}{Example}
\date{}
\numberwithin{equation}{section}
\numberwithin{theorem}{section}
\numberwithin{lemma}{section}
\numberwithin{corollary}{section}
\numberwithin{remark}{section}
\numberwithin{proposition}{section}
\numberwithin{definition}{section}

\def \Div {\mathrm{div}}

\def \R {\mathbb{R}}

\def \diam {\mathrm{diam}}

\def \dist {\mathrm{dist}}

\newcommand{\tr}{\mathrm{Tr}}
\newcommand{\defeq}{\mathrel{\mathop:}=}

\begin{document}

\title[Regularity for models with non-homogeneous degeneracy]{Geometric gradient estimates for fully nonlinear models with non-homogeneous degeneracy and applications}

\author[J.V. da Silva]{Jo\~{a}o Vitor da Silva}
\address{Departamento de Matem\'atica \hfill\break \indent Instituto de Ci\^{e}ncias Exatas \hfill\break \indent Universidade de Bras\'{i}lia
\hfill\break \indent Campus Universit\'{a}rio Darcy Ribeiro, CEP: 70910-900, Bras\'{i}lia - DF - Brazil.}
\email{J.V.Silva@mat.unb.br}

\author[G.C. Ricarte]{Gleydson C. Ricarte}
\address{Departamento de Matem\'atica \hfill\break \indent Universidade Federal do Cear\'{a}\hfill\break \indent Campus do Pici, CEP: 60455-760 Fortaleza-CE, Brazil}
\email{ricarte@mat.ufc.br}

\begin{abstract}
We establish sharp $C_{\text{loc}}^{1, \beta}$ geometric regularity estimates for bounded solutions of a class of fully nonlinear elliptic equations with non-homogeneous degeneracy, whose model equation is given by
$$
    \left[|Du|^p+\mathfrak{a}(x)|Du|^q\right]\mathcal{M}_{\lambda, \Lambda}^{+}(D^2 u)= f(x, u) \quad \text{in} \quad \Omega,
$$
for a bounded and open set $\Omega \subset \R^N$, and appropriate data $p, q \in (0, \infty)$,  $\mathfrak{a}$ and $f$. Such regularity estimates simplify and generalize, to some extent, earlier ones via totally different modus operandi. Our approach is based on geometric tangential methods and makes use of a refined oscillation mechanism combined with compactness and scaling techniques. In the end, we present some connections of our findings with a variety of nonlinear geometric free boundary problems and relevant nonlinear models in the theory of elliptic PDEs, which may have their own interest. We also deliver explicit examples where our results are sharp.

\bigskip

\noindent \textbf{Keywords:} Fully nonlinear elliptic problems, operators with non-homogeneous degeneracy, sharp H\"{o}lder gradient estimates.

\bigskip

\noindent \textbf{MSC (2010)}: 35B65, 35J60, 35J70, 35R35.
\end{abstract}

\maketitle

\section{Introduction}\label{s1}

In this work we shall derive sharp $C_{\text{loc}}^{1. \beta}$ geometric regularity estimates for solutions of a class of nonlinear elliptic equations having a non-homogeneous double degeneracy, whose mathematical model is given by
\begin{equation}\label{1.1}
   \mathcal{H}(x, Du)F(x, D^2 u) = f(x, u) \quad \text{in} \quad \Omega,
\end{equation}
for a $\beta \in (0, 1)$, a bounded and open set $\Omega \subset \R^N$, $f \in C^0(\Omega\times \R)\cap L^{\infty}(\Omega\times \R)$, where $F$ is assumed to be a second order, fully nonlinear (uniformly elliptic) operator, i.e., it is nonlinear in its highest (second) order derivatives.

Throughout this work we will suppose the following structural conditions:
\begin{itemize}
  \item[(A0)]({{\bf Continuity and normalization condition}})
  $$
  \text{Fixed}\,\, \Omega \ni x \mapsto F(x, \cdot) \in C^{0}(\text{Sym}(N)) \quad  \text{and} \quad F(\cdot, \text{O}_{N\times N}) = 0.
  $$
  \item[(A1)]({{\bf Uniform ellipticity}}) For any pair of matrices $\mathrm{X}, \mathrm{Y}\in Sym(N)$
$$
    \mathcal{M}_{\lambda, \Lambda}^-(\mathrm{X}-\mathrm{Y})\leq F(x, \mathrm{X})-F(x, \mathrm{Y})\leq\mathcal{M}_{\lambda, \Lambda}^+(\mathrm{X}-\mathrm{Y})
$$
where $\mathcal{M}_{\lambda, \Lambda}^{\pm}$ are the \textit{Pucci's extremal operators} given by
\[
   \mathcal{M}_{\lambda, \Lambda}^-(\mathrm{X})\defeq \lambda\sum_{e_i>0}e_i+\Lambda\sum_{e_i<0}e_i\quad\textrm{ and }\quad \mathcal{M}_{\lambda, \Lambda}^+(X)\defeq \Lambda\sum_{e_i>0}e_i+\lambda\sum_{e_i<0}e_i
\]
for \textit{ellipticity constants} $0<\lambda\leq \Lambda< \infty$, where $\{e_i(\mathrm{X})\}_{i=1}^{N}$ are the eigenvalues of $\mathrm{X}$.

Additionally, for our Theorem \ref{main1} (resp. Corollary \ref{Cormain1}), we must require some sort of continuity assumption on coefficients:

  \item[(A2)]({{\bf $\omega-$continuity of coefficients}}) There exist a uniform modulus of continuity $\omega: [0, \infty) \to [0, \infty)$ and a constant $\mathrm{C}_{\mathrm{F}}>0$ such that
$$
\Omega \ni x, x_0 \mapsto \Theta_{\mathrm{F}}(x, x_0) \defeq \sup_{\substack{\mathrm{X} \in  Sym(N) \\ \mathrm{X} \neq 0}}\frac{|F(x, \mathrm{X})-F(x_0, \mathrm{X})|}{\|\mathrm{X}\|}\leq \mathrm{C}_{\mathrm{F}}\omega(|x-x_0|),
$$
which measures the oscillation of coefficients of $F$ around $x_0$. For simplicity purposes, we shall often write $\Theta_{\mathrm{F}}(x, 0) = \Theta_{\mathrm{F}}(x)$.

Finally, for notation purposes we define
$$
\|F\|_{C^{\omega}(\Omega)} \defeq \inf\left\{\mathrm{C}_{\mathrm{F}}>0: \frac{\Theta_{\mathrm{F}}(x, x_0)}{\omega(|x-x_0|)} \leq \mathrm{C}_{\mathrm{F}}, \,\,\, \forall \,\,x, x_0 \in \Omega, \,\,x \neq x_0\right\}.
$$
\end{itemize}

In our studies, we enforce that the diffusion properties of the model \eqref{1.1} degenerate along an \textit{a priori} unknown set of singular points of solutions:
$$
   \mathcal{S}_0(u, \Omega^{\prime}) \defeq \{x \in \Omega^{\prime} \Subset \Omega: |Du(x)| = 0\}.
$$
Consequently, we will impose that $\mathcal{H}:\Omega \times \R^N \to [0, \infty)$ one behaves as
\begin{equation}\label{1.2}
     L_1 \cdot \mathcal{K}_{p, q, \mathfrak{a}}(x, |\xi|) \leq \mathcal{H}(x, \xi)\leq L_2 \cdot \mathcal{K}_{p, q, \mathfrak{a}}(x, |\xi|)
\end{equation}
for constants $0<L_1\le L_2< \infty$, where
\begin{equation}\label{N-HDeg}\tag{\bf N-HDeg}
   \mathcal{K}_{p, q, \mathfrak{a}}(x, |\xi|) \defeq |\xi|^p+\mathfrak{a}(x)|\xi|^q, \,\,\,\text{for}\,\,\, (x, \xi) \in \Omega \times \R^N.
\end{equation}

In turn, regarding the non-homogeneous degeneracy \eqref{N-HDeg}, we shall assume that the exponents $p, q$ and the modulating function $\mathfrak{a}(\cdot)$ fulfil
\begin{equation}\label{1.3}
   0< p \le q< \infty \qquad \text{and} \qquad \mathfrak{a} \in C^0(\Omega, [0, \infty)).
\end{equation}

One of the principal features of the model case \eqref{1.1} is its interplay between two distinct types of degeneracy rates, in accordance with the values of the modulating function $\mathfrak{a}(\cdot)$. For this reason, its diffusion process exhibits a kind of non-uniformly elliptic and doubly degenerate character, which mixes up two different $p-$degenerate type operators (cf. \cite{ART15}, \cite{BD1}, \cite{BD2}, \cite{BD3}, \cite{BDL19} and \cite{IS}).

Mathematically, \eqref{1.1} consists of a new model case of a fully nonlinear elliptic equation enjoying a non-homogeneous degenerate term, which constitutes a non-divergent counterpart of certain variational integrals of the calculus of variations with non-standard growth conditions as follows
\begin{equation}\label{DPF}\tag{\bf DPF}
 \displaystyle   \left(W_0^{1,p}(\Omega)+u_0, L^m(\Omega)\right) \ni (w, f) \mapsto \int_{\Omega} \left(\mathcal{K}_{p, q, \mathfrak{a}}(x, |Dw|)-fw\right)dx,
\end{equation}
where $\mathfrak{a}\in C^{0, \alpha}(\Omega,[0, \infty))$, for some $0< \alpha \leq 1$, $0<p\le q< \infty$ and $m \in (N, \infty]$, see \cite{BCM15III}, \cite{DeFM19I}, \cite{DeFM19II}, \cite{DeFM20}, \cite{DeFO} and \cite{LD18} for enlightening works.

In the light of recent research, one of our inspirations was the De Filippis' manuscript \cite{DeF20}, where $C_{\text{loc}}^{1, \gamma}-$regularity for viscosity solutions of
$$
  \left[|Du|^p+\mathfrak{a}(x)|Du|^q\right]F(D^2 u) = f \in L^{\infty}(\Omega), \quad (\text{with \eqref{1.3} in force}).
$$
was studied, for some $\gamma \in (0, 1)$ depending on universal parameters.

It is noteworthy to mention that our geometric approach is a byproduct of a new oscillation-type estimate (first considered in \cite{ATU17} and \cite{APR}, see also \cite{AdaSRT19}, \cite{ATU18}, \cite{daS19} and \cite{DSVII}) combined with a localized analysis, whose proof is conducted by studying two complementary cases:

\begin{enumerate}
  \item[\checkmark] If the gradient is small with a controlled magnitude, then a balanced perturbation of the $\mathfrak{F}-$harmonic profile leads to the inhomogeneous problem at the limit via a stability argument in a $C^1-$fashion.
  \item[\checkmark] On the other hand, if the gradient has a uniform lower bound, i.e. $|Du|\geq L_0>0$, then classical estimates (see, e.g. \cite{C89}, \cite{CC95} and \cite{Tru88}) can be enforced since the operator becomes uniformly elliptic:
  $$
  \mathcal{M}_{\lambda, \Lambda}^-(D^2 u)\leq C_0(L^{-1}_0, \|f\|_{L^{\infty}(\Omega)}) \,\, \text{and} \,\, \mathcal{M}_{\lambda, \Lambda}^+(D^2 u)\geq -C_0(L^{-1}_0, \|f\|_{L^{\infty}(\Omega)}).
  $$

\end{enumerate}
 Such a strategy relies on original ideas from \cite[Theorem 1]{ALS15} and \cite[Theorem 10]{LL17}. Finally, different from \cite[Theorem 3.1]{ART15}, \cite[Theorem 1.1]{APR}, \cite[Theorem 1]{DeF20} and \cite[Theorem 1]{IS}, our approach does not make use of an iterative argument of a suitable switched problem (a sort of deviation from planes), neither invokes a blow-up argument as the one presented in \cite[Lemma 9]{LL17}.

Although viscosity solutions of \eqref{1.1} under structural assumptions (A0)-(A2), \eqref{1.2} and \eqref{1.3} are known to be locally of the class $C^{1, \gamma}$, such an optimal exponent was not addressed in such a study. For this reason, our main purpose will be obtaining such sharp estimates via an alternative, and totally different, geometric approach. In addition, we stress that such a quantitative information plays a decisive role in the development of several analytic and geometric problems in pure and applied mathematics, such as in the study of blow-up analysis, related weak geometric and free boundary problems, and for providing some Liouville type results, see \cite{ALS15}, \cite{BD0}, \cite{daSLR20}, \cite{daSO19}, \cite{daSOS18}, \cite{daSRS19I}, \cite{daSS18}, \cite{DSVI}, \cite{DSVII}, \cite{T15} and \cite{Tei16} for some enlightening examples.

\subsection{Statement of the main results}\label{ssec.def}

We will present some definitions, as well as some useful auxiliary results for our approach in this section.
Now, let us introduce the notion of viscosity solution for our operators.

\begin{definition}
  [{\bf Viscosity solutions}] A function $u \in C^{0}(\Omega)$ is a viscosity super-solution (resp. sub-solution) to \eqref{1.1} if whenever $\varphi \in C^2(\Omega)$ and $x_0 \in \Omega$ such that $u-\varphi$ has a local minimum (resp. a local maximum) at $x_0$, then
$$
   \mathcal{H}(x_0, D\varphi(x_0))F(x_0, D^2 \varphi(x_0)) \leq f(x_0, \varphi(x_0)) \qquad \text{resp.} \,\,\,(\cdots \geq f(x_0, \varphi(x_0)))
$$
Finally, $u$ is said to be a viscosity solution if it is simultaneously a viscosity super-solution and
a viscosity sub-solution.
\end{definition}

In order to measure the smoothness of solutions in suitable spaces, we are going to use the following norms and semi-norms (see, \cite[Section 1]{Kov99}):

\begin{definition}
  [{\bf $C^{1, \alpha}$ norm and semi-norm}] For $\alpha \in (0,1]$, $C^{1, \alpha}(\Omega^{\prime})$ denotes the space of functions $u$ whose spacial gradient $Du(x)$ there exists in the classical sense for every $x\in \Omega^{\prime} \Subset \Omega$ such that
$$
\|u\|_{C^{1, \alpha}(\Omega^{\prime})} \defeq \|u\|_{L^{\infty}(\Omega^{\prime})} + \|Du\|_{L^{\infty}(\Omega^{\prime})} + [u]_{C^{1, \alpha}(\Omega^{\prime})}
$$
is finite, where we have the semi-norm
\begin{equation}\label{Semi-Norm}
  [u]_{C^{1, \alpha}(\Omega^{\prime})} \defeq \sup_{\substack{x_0\in \Omega^{\prime} \\0< r\leq \diam(\Omega)}}  \inf_{\mathfrak{l \in \mathcal{P}_1}} \frac{\|u-\mathfrak{l}\|_{L^{\infty}(B_r(x_0)\cap \Omega^{\prime})}}{r^{1+\alpha}},
\end{equation}
where $\mathcal{P}_1$ means the spaces of affine functions. Hence, $u \in C^{1, \alpha}(\Omega^{\prime})$ implies every component of $Du$ is $C^{0, \alpha}(\Omega^{\prime})$ (see, \cite[Main Theorem]{Kov99}).
\end{definition}

Now we are in a position to state our main results. The first one establishes an optimal geometric estimate. In effect, it reads that if the source term is bounded and (A0)-(A2), \eqref{1.2} and \eqref{1.3} are in force, then any bounded viscosity solution of \eqref{1.1} belongs to $C^{1, \beta}$ at interior points, where
 \begin{equation}\label{SharpExp}
   \beta \in (0, \alpha_{\mathrm{F}}) \cap \left(0, \frac{1}{p+1}\right],
 \end{equation}
where $\alpha_{\mathrm{F}} = \alpha_{\mathrm{F}}(N, \lambda, \Lambda) \in (0, 1]$ is the optimal exponent to (local) H\"{o}lder continuity of gradient of solutions to homogeneous problem with ``frozen coefficients'' $F(D^2 \mathfrak{h}) = 0$ (see, \cite{C89}, \cite[Ch.5 \S 3]{CC95} and \cite{Tru88}).  We will present some examples where such an exponent is precisely $\alpha_{\mathrm{F}}=1 $ in Section \ref{Examples}.

\begin{theorem}[{\bf Sharp regularity at interior points}]\label{main1} Assume that assumptions (A0)-(A2), \eqref{1.2} and \eqref{1.3} there hold. Let $u$ be a bounded viscosity solution to \eqref{1.1} with  $f \in L^\infty(\Omega\times \R)$. Then, $u$ is $C^{1, \beta}$, at interior points, for $\beta$ satisfying \eqref{SharpExp}. More precisely, for any point $x_0 \in \Omega^{\prime}\Subset \Omega$ there holds
$$
       \displaystyle [u]_{C^{1, \beta}(B_r(x_0))}\leq C\cdot\left(\|u\|_{L^{\infty}(\Omega)} +1+ \|f\|_{L^{\infty}(\Omega\times \R)}^{\frac{1}{p+1}}\right),
$$
for $0<r < \frac{1}{2}$ where $C>0$ is a universal constant\footnote{A constant is said to be universal if it depends only on dimension, degeneracy and ellipticity constants, $\alpha_{\mathrm{F}}$, $\beta$, $L_1, L_2$ and $\|F\|_{C^{\omega}(\Omega)}$.}
\end{theorem}

As a result of Theorem \ref{main1} we have the following (cf. \cite[Corollary 3.2]{ART15}):

\begin{corollary}\label{Cormain1} Assume that the assumptions of Theorem \ref{main1} are in force. Suppose $F$ to be one of the operators presented in Section \ref{Examples}. Then, $u$ is $C_{\text{loc}}^{1, \frac{1}{p+1}}(\Omega)$. Moreover, there holds
$$
       \displaystyle [u]_{C^{1, \frac{1}{p+1}}(B_r(x_0))}\leq C\cdot\left(\|u\|_{L^{\infty}(\Omega)} +1+ \|f\|_{L^{\infty}(\Omega\times \R)}^{\frac{1}{p+1}}\right).
$$
\end{corollary}

As another consequence of Theorem \ref{main1}, we also derive the sharp gradient growth-rate at interior points.

\begin{corollary}[{\bf Sharp gradient growth}]\label{Cormain2} Under the assumptions of Theorem \ref{main1},  there exists a universal constant $C>0$ such that for any point $x_0 \in B_{1/2}(0)$ and for all $0<r < \frac{1}{4}$
$$
   \displaystyle \sup_{B_r(x_0)} \frac{|D u(x)-Du(x_0)|}{r^{\beta}} \leq  C\cdot\left(\|u\|_{L^{\infty}(\Omega)} +1+ \|f\|_{L^{\infty}(\Omega\times \R)}^{\frac{1}{p+1}}\right).
$$
\end{corollary}

From now on, we will label the critical zone of existing solutions as
$$
\mathcal{S}_{r, \beta}(u, \Omega^{\prime}) \defeq \left\{x_0 \in \Omega^{\prime} \Subset \Omega: |Du(x_0)| \leq r^{\beta}, \,\,\,\mbox{for}\,\,\,0\leq r\ll1\right\}.
$$

A geometric interpretation to Theorem \ref{main1} says that if $u$ solves \eqref{1.1} and $x_0 \in \mathcal{S}_{r, \beta}(u, \Omega^{\prime})$, then near $x_0$ we obtain
$$
    \displaystyle \sup_{B_r(x_0)} |u(x)|\leq |u(x_0)|+C\cdot r^{1+\beta},
$$

On the other hand, from a geometric viewpoint, it is a pivotal qualitative information to obtain the (counterpart) sharp lower bound estimate for such operators with non-homogeneous degeneracy. Such a feature is denominated \textit{Non-degeneracy property} of solutions.

Therefore, under a natural, non-degeneracy assumption on $f$, we further obtain the precise behavior of solutions at certain interior singular regions.

\begin{theorem}[{\bf Non-degeneracy estimate}]\label{main3} Suppose that the assumptions of Theorem \ref{main1} are in force. Let $u$ be a bounded viscosity solution to \eqref{1.1} with $f(x, u)=f(x) \geq \mathfrak{m}>0$ in $\Omega$. Given $x_0 \in \mathcal{S}_{r, \beta}(u, \Omega^{\prime})$, there exists a constant $\mathfrak{c} = \mathfrak{c}(\mathfrak{m},\|\mathfrak{a}\|_{L^{\infty}(\Omega)}, L_1, N, \lambda, \Lambda, p,q, \Omega)>0$, such that
$$
  \displaystyle \sup_{\partial B_r(x_0)} u(x) \geq u(x_0) + \mathfrak{c} \cdot r^{1+\frac{1}{p+1}}\quad \text{for all} \quad  0<r < \frac{1}{2}.
$$
\end{theorem}

Our findings extend/generalize regarding non-variational scenario, former results (H\"{o}lder gradient estimates) from \cite[Theorem 3.1 and Corollary 3.2]{ART15}, \cite{BD1}, \cite{BD2}, \cite{BD3} and \cite{IS}, and to some extent, those from \cite[Theorem 1]{DeF20} by making using of different approaches and techniques adapted to the general framework of the fully nonlinear non-homogeneous degeneracy models.

Additionally, one of the main novelties of our approach consists of removing the restriction of analyzing $C_{\text{loc}}^{1, \alpha}$ regularity estimates just along the \textit{a priori} unknown set of singular points of solutions $\mathcal{S}_{0}(u, \Omega)$, where the ``ellipticity of the operator'' degenerates (see, for example, \cite{daSLR20}, \cite{daSO19}, \cite{daSOS18}, \cite{daSRS19I}, \cite{daSS18}, \cite{T14}, \cite{T15} and \cite{Tei16}, where improved regularity estimates were addressed along certain sets of degenerate points of existing solutions). Finally, they are even striking for the simplest toy model:
$$
 u\mapsto  \left[|Du|^p+\mathfrak{a}(x)|Du|^q\right]\Delta u  \quad (\text{a bounded mapping with \eqref{1.3} in force}).
$$

\subsection{Motivation: nonlinear variational/non-variational problems}

Regularity estimates for minimizers of integral functionals exhibiting degeneracy of double phase type and  non-uniform ellipticity feature have been the focus of intensive investigation over the last decades by several authors. A model case in question is given by the double phase functional \eqref{DPF}, whose studies date back to Zhikov's fundamental works in the context of Homogenization problems and Elasticity theory, as well as to give new instances of the occurrence of Lavrentiev phenomenon, see \textit{e.g.} \cite{Zhi93} and \cite{Zhi95}.

Notice that the Euler–Lagrange equation to \eqref{DPF} exhibits a type of non-uniform and doubly degenerate ellipticity, which mixes up two different kinds of $p-$Laplacian type operators:
$$
   - \Div(\mathcal{A}(x,\nabla u)) = f(x) \qquad \text{with}\qquad \mathcal{A}(x,\xi) \defeq p|\xi|^{p-2}\xi +q\mathfrak{a}(x)|\xi|^{q-2}\xi.
$$

By way of a reminder, it should be noted that minimizers to \eqref{DPF} play an important role in some contexts of Materials Science and engineering, where they describe the behavior of certain strongly anisotropic materials, whose hardening estates, connected to the gradient growth exponents, change point-wisely. Specifically, a mixture of two heterogeneous materials, with hardening $(p, q)$-exponents, can be performed according to the intrinsic geometry of the null set of the modeling coefficient $\mathfrak{a}(\cdot)$.

In this direction, recently Colombo and Mingione in the seminal works \cite{CM15} and \cite{CM15II} have addressed regularity issues of local minimizers $u:\Omega \to \R$ of a general class of variational integrals whose model functional is given \eqref{DPF} (with $f\equiv 0$), where
$$
   1<p<q\leq p+\alpha, \quad 0\le \mathfrak{a} \in C^{0,\alpha}(\Omega, [0, \infty)), \quad 0<\alpha \le 1.
$$
In such a context, they prove for local minimizers $u \in W^{1,p}(\Omega)\cap L^{\infty}(\Omega)$ the regularity of the gradient $Du \in C^{0, \beta}_{\text{loc}}(\Omega)$ for a universal constant $\beta \in (0,1)$, depending only on $N,p,q$  and $\alpha$, see \cite[Theorem 1.1]{CM15} and \cite[Theorem 1.2]{CM15II}. Additionally, elliptic type models as \eqref{DPF} and related variational obstacle problems, are treated under optimal assumptions on the modulating coefficient $\mathfrak{a}(\cdot)$, on the source term $f(\cdot)$ and obstacle in the recent paper \cite{DeFM20}, where the authors address optimal Lipschitz bounds for minimizers. Finally, we must also quote the fundamental works \cite{BCM15III}, \cite{DeFM19I}, \cite{DeFM19II}, \cite{DeFO} and \cite{LD18}, which addressed existence, regularity and multiplicity results for certain homogeneous/inhomogeneous double phase type problems like \eqref{DPF}.

Regarding non-variational scenario, i.e., fully nonlinear models such as \eqref{1.1}, their regularity properties have been the object of intense investigation over the last decade due to their proper connection to several issues arising in pure mathematics, as well as a variety of geometric and free boundary problems (see \textit{e.g.} \cite{ART15}, \cite{ART17}, \cite{BD1}, \cite{BD2}, \cite{BD3}, \cite{BDL19}, \cite{daSLR20}, \cite{DSVI}, \cite{DSVII}, \cite{DeF20} and \cite{IS} for an extensive but incomplete list of the latest contributions).

In turn, the simplest model with single degeneracy structure is given by
$$
   \mathcal{G}(x, Du, D^2 u) \defeq   |Du|^pF(x, D^2 u) \quad \text{with} \quad 0< p< \infty.
$$
At this point, its regularity properties (see \textit{e.g} \cite[Theorem 3.1]{ART15}, \cite[Theorem 1.1]{BD1}, \cite[Theorem 1.1]{BD2}, \cite[Theorem 1.1]{BDL19} and \cite[Theorem 1]{IS}) are intrinsically connected to three important aspects of the model equation:

\begin{enumerate}
  \item[(1)] {\bf Estimate to homogeneous problem with ``frozen coefficients''};
       $$
       \begin{array}{ccc}
          &  & \mathfrak{h} \in C_{\text{loc}}^{1, \alpha_{\mathrm{F}}}(\Omega) \\
         F(x_0, D^2 \mathfrak{h}) = 0 \,\, \text{in} \,\, \Omega & \stackrel[]{\text{A priori estimate}}{\Longrightarrow}  & \text{and} \\
          &  & \|\mathfrak{h}\|_{C^{1, \alpha_{\mathrm{F}}}(\Omega^{\prime})} \leq C(N, \lambda, \Lambda
         ) \cdot \|\mathfrak{h}\|_{L^{\infty}(\Omega)}
       \end{array}
       $$
  \item[(2)] {\bf Degeneracy degree of the model}:
  $$
  |Du|^p\mathcal{M}_{\lambda, \Lambda}^{-}(D^2 u) \leq \mathcal{G}(x, Du, D^2 u) \le|Du|^p\mathcal{M}_{\lambda, \Lambda}^{+}(D^2 u).
  $$
  \begin{center}
  	\text{More $p-$degeneracy degree} $\Longrightarrow$  \text{Loss of regularity} (cf. \cite{BdaSR19} and \cite{daSRS19I})
  \end{center}
  \item[(3)] {\bf Integrability properties of the source term}:
     $$
      \mathcal{G}(x, Du, D^2 u) = f \in L^m(\Omega)\cap C^0(\Omega) \quad \text{for} \quad N< m\leq \infty.
     $$
     $$
     \text{Enough m-integrability} \,\,\,\Rightarrow \,\,\, \text{More regularity in a suitable functional space}.
     $$
{\small{
$$
\left\{
       \begin{array}{ccc}
\left\|\frac{\partial^2 u}{\partial x_i \partial x_j}\right\|_{L^{ m}(\Omega^{\prime})} & \stackrel[\leq]{\text{Calder\'{o}n-Zygmund's}}{{\small\text{estimate}\, (p\approx 0)}} &  C(\lambda, \Lambda, N, m) \cdot \left(\|u\|_{L^{m}(\Omega)} + \|f\|_{L^{m}(\Omega)}\right) \\
          & \text{or}  &  \\
         \left\|\frac{\partial u}{\partial x_i}\right\|_{C^{0, \alpha}(\Omega^{\prime})} & \stackrel[]{\text{Gradient estimate}}{\leq} &  C(\lambda, \Lambda, N, m, \alpha) \cdot \left(\|u\|_{L^{\infty}(\Omega)} + \|f\|^{\frac{1}{p+1}}_{L^{m}(\Omega)}\right)
       \end{array}
\right.
$$}}
\end{enumerate}

Therefore, after these important breakthroughs and taking into account the previous highlights, we have decided to develop a geometric analysis of the behavior of solutions of \eqref{1.1} at interior points, in order to obtain an optimal $C_{\text{loc}}^{1, \beta}$ regularity estimate and deliver some relevant applications.

It is worth mentioning that under appropriate structural conditions, the techniques used throughout this manuscript allow us to treat the case of more general elliptic variational problems as follows
$$
 \left(u_0+W_0^{1,p}(\Omega), L^m(\Omega)\right) \ni (w, f) \mapsto \min \int_{\Omega} \left(\mathcal{H}_0(x,\nabla w)-fw\right)dx,
$$
as long as the integrand has suitable double phase type structure:
$$
   L_1 \cdot \left(\frac{1}{p}|\xi|^p+\mathfrak{a}(x)\frac{1}{q}|\xi|^q\right)\leq \mathcal{H}_0(x,\xi)\leq L_2 \cdot \left(\frac{1}{p}|\xi|^p+\mathfrak{a}(x)\frac{1}{q}|\xi|^q\right),
$$
where $0<L_1 \le L_2< \infty$,   $1<p<q<\infty$ and $m \in (N, \infty]$, which make possible to access available existence/regularity results for weak solutions (see e.g., \cite{BCM15III}, \cite{CM15}, \cite{CM15II}, \cite{DeFM19I}, \cite{DeFM19II}, \cite{DeFM20}, \cite{DeFO} and \cite{LD18} for some pivotal references).

We also stress that our geometric approach is particularly refined and quite far-reaching in order to be employed in other classes of problems. Indeed, we believe that it can be also extended to fully nonlinear elliptic equations with non-homogeneous term of generalized Double Phase type, see \cite{ByOh20}, or of Infinity Multi-Phase type, cf. \cite[Section 2]{HaOk19} just to cite a few.

Finally, an extension of our results also holds to problems modelled by
$$
   \left(|Du|^p + \sum_{i=1}^{N} \mathfrak{a}_i(x)|Du|^{q_i}\right)F(x, D^2 u) = f(x, u) \quad \text{in} \quad \Omega,
$$
where $0\le\mathfrak{a}_i \in C^0(\Omega)$, $i \in \{1, \cdots, N\}$, and $0<p\leq q_1\leq \cdots\leq q_N< \infty$.

\subsection{Our strategy and main difficulties of the model problem}

The main idea behind the proof of Theorem \ref{main1} is to perform a geometric decay argument along those points around which the equation degenerates, i.e. where the gradient becomes very small, which represents a totally different approach in contrast with former ones \cite{ART15}, \cite{APR}, \cite{AR18} and \cite{DeF20} just to cite a few.

In the sequel, the purpose will be to make use of an $\mathfrak{F}-$harmonic approximation in a $C^1-$fashion (Lemma \ref{lem.flat}) to ensure that viscosity solutions are ``geometrically close'' to their tangent plane in a suitable manner, i.e.
{\small{
$$
C^1-\text{closeness}\stackrel[\Longrightarrow]{\text{Geometric estimate}}{}    \,\,\displaystyle\sup_{B_{\rho}(x_0)} \frac{\left|u(x)-u(x_0)-D u(x_0)\cdot (x-x_0)\right|}{\rho^{1+\beta}}\leq 1,
$$}}
thereby getting a geometric estimate, in effect, the first step in an iteration process, see Corollary \ref{c3.1} for further details. Moreover, different from second order operators with linear first order terms, i.e. $F(x, Du, D^2 u) = f(x)$ with
$$
\begin{array}{rcl}
    \mathcal{M}_{\lambda, \Lambda}^-(\mathrm{X}-\mathrm{Y})-L|\xi-\varsigma| & \leq & F(x, \xi, \mathrm{X})-F(x, \varsigma, \mathrm{Y}) \\
   & \leq & \mathcal{M}_{\lambda, \Lambda}^+(\mathrm{X}-\mathrm{Y})+L|\xi-\varsigma|,
\end{array}
$$
one complication we are dealing with is that we can no longer proceed with an iterative scheme as the ones used in \cite[Theorem 1.1]{daS19} or \cite[Section 5]{DT17}, i.e.
$$
 \displaystyle\sup_{B_{\rho^k}(x_0)} \frac{\left|u(x)-\mathfrak{l}_k(x)\right|}{\rho^{k(1+\beta)}}\leq 1 \quad \stackrel[\Longrightarrow]{\text{Dini-Campanato}}{{\small\text{embedding}}} \,\,\,u \,\, \text{is} \,\,\,C^{1, \beta} \quad \text{at}\,\,\,x_0,
$$
because \textit{a priori} we do not know the equation which is satisfied by
$$
    B_1(0) \ni x \mapsto \frac{(u-\mathfrak{l}_k)(\rho^kx)}{\rho^{k(1+\beta)}}, \quad \text{for}\quad \{\mathfrak{l}_k\}_{k\in\mathbb{N}}\quad \text{a sequence of affine functions},
$$
since $v \mapsto  \mathcal{H}(x, D v)F(x, D^2 v)$ is not invariant by affine mappings. Nevertheless, it provides  quantitative information on the oscillation of $u$:
$$
\displaystyle \sup_{B_{\rho}(x_0)}\frac{\rho^{-1}\left|u(x)-u(x_0)\right|}{\rho^{\beta}+|D u(x_0)|}\leq1   \stackrel[\Longrightarrow]{\text{Iteration}}{} \sup_{B_{\rho^k}(x_0)}\frac{\rho^{-k}\left|u(x)-u(x_0)\right|}{\rho^{k\beta}+\frac{|D u (x_0)|(1-\rho^{(k-1)\beta})}{1-\rho^{\beta}}}\leq 1,
$$
which proves to be the proper estimate for continuing with an iterative process, provided we get a sort of suitable control under the magnitude of the gradient (point-wisely) (see, Lemmas \ref{lem.dy} and \ref{l3.3} for more details).

In conclusion, our manuscript is organized as follows: in Section \ref{sec.ALemmas} we show that solutions of \eqref{1.1} can be approximated by $\mathfrak{F}$-harmonic functions in a suitable $C^1-$fashion. In Section \ref{sec.ProofMRes} we prove the main results of this paper (Theorems \ref{main1}, \ref{main3} and Corollaries \ref{Cormain1} and \ref{Cormain2}) by obtaining sharp regularity estimates at interior points. Section \ref{sec.Applic} establishes connections with some nonlinear geometric free boundary problems and beyond. Section \ref{Sec.FurtApplic} employs some interesting applications in the study of some nonlinear problems in the theory of elliptic PDEs. Finally, Section \ref{Examples} is devoted to presenting a number of settings where our results are explicit.

\section{Auxiliary Lemmatas}\label{sec.ALemmas}

We start this section by presenting some known results that will be used later on. The first one is a kind of ``Cutting Lemma'', which strongly relies on \cite[Lemma 6]{IS} and it is concerned with the homogeneous doubly degenerate problem. We refer the reader to \cite[Lemma 4.1]{DeF20}, which, on the one hand, is not the precise statement of such a result, but on the other hand, it can be inferred
from \cite[Lemma 4.1]{DeF20} (see also \cite[Lemma 6]{IS}) by a careful inspection of the proof.

\begin{lemma}[{\bf Cutting Lemma}]\label{cutting}
Let $F$ be an operator satisfying (A0)-(A2) and $u$ be a viscosity solution of
\[
\mathcal{H}(x, Du)F(x, D^2u)=0 \quad \textrm{ in }\quad B_1(0).
\]
Then $u$ is viscosity solution of
\[
   F(x,D^2u)=0 \quad \textrm{ in }\quad B_1(0).
\]
\end{lemma}

An essential tool we will use is the fundamental gradient estimate from \cite[Theorem 1]{DeF20}, which we will state below for completeness.

\begin{theorem}[{\bf Gradient estimates}]\label{GradThm} Let $F$ be an operator satisfying (A0)-(A2) and let $u$ be a bounded viscosity solution to
$$
 \left[|Du|^p +\mathfrak{a}(x)|Du|^q\right] F(D^2u) = f \in L^{\infty}(B_1(0)).
$$
Then,
\begin{equation*}
\|u\|_{C^{1,\gamma}\left(B_{\frac{1}{2}(0)}\right)}\leq C\cdot \left(\|u\|_{L^{\infty}(B_1(0))} +1+ \|f\|_{L^{\infty}(B_1(0))}^{\frac{1}{p+1}}\right)
\end{equation*}
for universal constants $\gamma \in (0, 1)$ and $C>0$.
\end{theorem}

In the sequel, we define an appropriate class of solutions to our problem:

\begin{definition}
  \label{FineClass} Let $F$ be a fully nonlinear operator satisfying (A0)-(A2), For $\mathfrak{a} \in C^0(\Omega, [0, \infty))$ and $f\in L^\infty(B_1(0)\times \R)$, we say that $u \in \mathcal{J}(F, \mathfrak{a}, f)(B_1(0))$ if
\begin{itemize}
  \item $\mathcal{H}(x, Du)F(x, D^2 u) = f(x, u)$ in $B_1(0)$ in the viscosity sense.
  \item  $ \|u\|_{L^{\infty}(B_1(0))}\leq 1$ in $B_1(0)$.
\end{itemize}
\end{definition}

Before presenting our main Key Lemma, the following simple stability result will be instrumental in the proof of Lemma \ref{lem.flat}:

\begin{lemma}\label{lem.stab} Let $\{F_k(x,X)\}_{k\in \mathbb{N}}$ be a sequence of operators satisfying (A0)-(A2) with the same ellipticity constants and same modulus of continuity in $B_1(0)$. Then there exists an elliptic operator $F_0$ which still satisfies (A0)-(A2) such that
\[
   F_k\longrightarrow F_0 \quad \mbox{uniformly on compact subsets of} \quad \text{Sym}(N)\times B_1(0).
\]
\end{lemma}

The first key step towards the proof of Theorem \ref{main1} is to show that the solutions to \eqref{1.1}, in inner domains, can be approximated in a suitable manner by $\mathfrak{F}-$harmonic profiles, and that during such a process certain regularity properties of solutions are preserved.

At this point we are in a position to prove the following Key Lemma (cf. \cite[Lemma 5.1]{ART15}, \cite[Lemma 2.6]{DT17} and \cite[Lemma 4.1]{DeF20}):

\begin{lemma}[{\bf Approximation by $\mathfrak{F}-$harmonic solutions}]\label{lem.flat}
Let $\mathfrak{h} \in C^{0}\left(\overline{B_{\frac{1}{2}}(0)}\right)$ be the unique viscosity solution to
$$
\left\{
\begin{array}{rcrcl}
  F(D^2 \mathfrak{h}) & = & 0 & \text{in} & B_{\frac{1}{2}}(0) \\
  \mathfrak{h} & = & u & \text{on} & \partial B_{\frac{1}{2}}(0),
\end{array}
\right.
$$
Then, given $0<\iota<1$, there exists a $\delta= \delta(\iota, N, \lambda, \Lambda, p, q)>0$ such that if $u \in \mathcal{J}(F, \mathfrak{a}, f)(B_{1}(0))$ with
$$
    \max\left\{\Theta_{\mathrm{F}}(x), \left\|f\right\|_{L^{\infty}(B_{1}(0)\times \R)}\right\} \leq \delta
$$
then
\begin{equation}\label{eq.flat}
     \max\left\{\|u_k-\mathfrak{h}\|_{L^{\infty}\left(B_{\frac{1}{2}}(0)\right)}, \|Du_k-D\mathfrak{h}\|_{L^{\infty}\left(B_{\frac{1}{2}}(0)\right)}\right\} \leq \iota.
\end{equation}
\end{lemma}

\begin{proof} The proof is based on a \textit{reductio ad absurdum} argument. For this end, suppose that the lemma does not hold. This means that for some $\iota_0\in (0, 1)$ we could find sequences $\{u_k\}_k$, $\{\mathfrak{h}_k\}_k$ $\{F_k\}_k$, $\{\mathfrak{a}_k\}_k$ and $\{f_k\}_k$ satisfying:
\begin{itemize}
\item $u_k\in \mathcal{J}(F_k, \mathfrak{a}_k, f_k)(B_{1}(0))$;
\item $\max\left\{\Theta_{\mathrm{F}_k}(x), \|f_k\|_{L^{\infty}(B_{1}(0)\times \R)}\right\} = \text{o}(1) \quad \text{when} \quad k\gg 1$;
\end{itemize}
and
$$
\left\{
\begin{array}{rclcl}
  F_k(D^2 \mathfrak{h}_k) & = & 0 & \text{in} & B_{\frac{1}{2}}(0) \\
  \mathfrak{h}_k & = & u_k & \text{on} & \partial B_{\frac{1}{2}}(0),
\end{array}
\right.
$$
in the viscosity sense. However,
  \begin{equation}\label{Eqcont}
     \max\left\{\|u_k-\mathfrak{h}_k\|_{L^{\infty}\left(B_{\frac{1}{2}}(0)\right)}, \|Du_k-D\mathfrak{h}_k\|_{L^{\infty}\left(B_{\frac{1}{2}}(0)\right)}\right\}  > \iota_0 \quad \forall\,\,\, k \in \mathbb{N}.
  \end{equation}

From Maximum Principle we obtain
$$
\|\mathfrak{h}_k\|_{L^{\infty}\left(B_{\frac{1}{2}}(0)\right)} \leq \|u_k\|_{L^{\infty}\left(\partial B_{\frac{1}{2}}(0)\right)} \leq 1
$$
By definition, we have
$$
       \mathcal{H}_k(x, Du_k)F_k(x, D^2 u_k) = f_k(x, u_k) \quad \text{and} \quad  \|u_k\|_{L^{\infty}\left(B_1(0)\right)}\leq1.
$$
Hence, by H\"{o}lder regularity of solutions (see,  \cite[Proposition 4.10]{CC95} and \cite[Theorem 2]{DeF20}), up to a subsequence, $\mathfrak{h}_k \to \mathfrak{h}_0$ and $u_k \to u_0$ uniformly in $\overline{B_{\frac{1}{2}}(0)}$. Furthermore, $\mathfrak{h}_{k} \to  \mathfrak{h}_0$ local uniformly in the $C^{1}-$topology (see \textit{e.g.} \cite{C89}, \cite[Section 5.3]{CC95} and \cite{Tru88}). Now, from \cite[Theorem 1]{DeF20} we can estimate:
{\small{
$$
  \|Du_k\|_{C^{1, \gamma}\left(B_{\frac{1}{2}}(0)\right)}, [u_k]_{C^{1, \gamma}\left(B_{\frac{1}{2}}(0)\right)} \leq C\cdot \left(\|u_k\|_{L^{\infty}(B_{1}(0))}+1+\|f_k\|^{\frac{1}{p+1}}_{L^{\infty}(B_{1}(0)\times \R)}\right)
$$}}
for some $\gamma \in (0, 1)$ and a universal constant $C>0$. Thus, up to a subsequence, $Du_k \to Du_0$ uniformly in $\overline{B_{\frac{1}{2}}(0)}$. In particular,  we conclude that
\begin{equation}\label{eq.cont}
\max\left\{\|u_0-\mathfrak{h}_0\|_{L^{\infty}\left(B_{\frac{1}{2}}(0)\right)}, \|Du_0-D\mathfrak{h}_0\|_{L^{\infty}\left(B_{\frac{1}{2}}(0)\right)}\right\}\geq \iota_0.
\end{equation}

On the other hand, from Lemma \ref{lem.stab}, there exists an elliptic operator $\mathfrak{F}_0$ satisfying (A0)-(A2) (with $\omega \equiv 0$) such that $F_k \to \mathfrak{F}_0$ locally uniformly in $\text{Sym}(N)$ for all $x \in B_1(0)$ fixed.

Now, by arguing as \cite[Lemma 4.1]{DeF20} and making use of stability results for viscosity solutions (cf. \cite[Corollary 2.7 and Remark 2.8]{BD2}), and from the ``Cutting Lemma'' \ref{cutting} we conclude that
\begin{equation}\label{EqDirProb}
  \left\{
\begin{array}{rclcl}
  \mathfrak{F}_0(D^2 \mathfrak{h}_0) & = & 0 & \text{in} & B_{\frac{1}{2}}(0) \\
  \mathfrak{h}_0 & = & u_0 & \text{on} & \partial B_{\frac{1}{2}}(0)\\
  \mathfrak{F}_0(D^2 u_0) & = & 0 & \text{in} & B_{\frac{1}{2}}(0).
\end{array}
\right.
\end{equation}
in the viscosity sense.

Finally, from the uniqueness of viscosity solutions to the Dirichlet problem \eqref{EqDirProb} we conclude that $u_0=\mathfrak{h}_0$ (see, \cite[Section 5.2]{CC95}), which clearly yields a contradiction with \eqref{eq.cont}. This concludes the proof.
\end{proof}

\begin{remark}[{\bf Smallness regime}]\label{SmallRegime} Let us argue on the scaling character of our problem which enables us to put the proof of Theorem \ref{main1} under the assumptions of Approximation Lemma \ref{lem.flat}. Let $u$ be a viscosity solution of \eqref{1.1}. Fix a point $x_0 \in \Omega^{\prime} \Subset \Omega$, we define $v: B_1(0) \to \R$ as follows
$$
    v(x) = \frac{u(\tau x+ x_0)}{\kappa}
$$
for parameters $\kappa, \tau>0$ to be determined later. It is easy to verify that $v$ fulfills (in the viscosity sense)
$$
\mathcal{H}_{\kappa, \tau}(x, D v) F_{\kappa, \tau}(x, D^2v) =  f_{\kappa, \tau}(x)  \textrm{ in } B_1(0),
$$
where
\[
\left\{
\begin{array}{rcl}
  F_{\kappa, \tau}(x, X) & \defeq & \frac{\tau^{2}}{\kappa}F\left(x_0+\tau x, \frac{\kappa}{\tau^{2}} X\right) \\
  f_{\kappa, \tau}(x, s) & \defeq & \frac{\tau^{p+2}}{\kappa^{p+1}}f(x_0+\tau x, \kappa s)\\
  \mathfrak{a}_{\kappa, \tau}(x) & \defeq & \left(\frac{\tau}{\kappa}\right)^{p-q}\mathfrak{a}(x_0+\tau x)\\
  \mathcal{H}_{\kappa, \tau}(x, \xi) &\defeq & \left(\frac{\tau}{\kappa}\right)^p\mathcal{H}\left(x_0+\tau x, \frac{\kappa}{\tau}\xi\right)\\
  \mathcal{K}_{p, q, \mathfrak{a}}^{\kappa, \tau}(x, |\xi|)&\defeq& |\xi|^p+\mathfrak{a}_{\kappa, \tau}(x)|\xi|^q.
\end{array}
\right.
\]
Hence, $F_{\kappa, \tau}$ fulfils the structural assumptions (A0) and (A1). Moreover,
$$
L_1 \cdot \mathcal{K}_{p, q, \mathfrak{a}}^{\kappa, \tau}(x, |\xi|)\leq  \mathcal{H}_{\kappa, \tau}(x, \xi) \leq L_2 \cdot \mathcal{K}_{p, q, \mathfrak{a}}^{\kappa, \tau}(x, |\xi|) \quad \text{for} \quad (x, \xi) \in \Omega \times \R^N.
$$
Now, for given $\iota \in (0, 1)$, which will be sufficiently small but fixed, let $\delta_{\iota}>0$ be the universal constant in the statement of Approximation Lemma \ref{lem.flat}. Then, we choose
$$
 \kappa \defeq \|u\|_{L^{\infty}(\Omega)} + 1 +\delta_{\iota}^{-1}\|f\|^{\frac{1}{p+1}}_{L^{\infty}(\Omega\times \R)}\\
$$
and
$$
 \tau = \min \left\{\frac{1}{2},\, \frac{1}{4}\dist(\Omega^{\prime},\, \partial \Omega), \left(\frac{\delta_{\iota}}{\|f\|_{L^{\infty}(\Omega \times \R)}+1}\right)^{\frac{1}{p+2}},\, \omega^{-1}\left(\frac{\delta_{\iota}}{\mathrm{C}_{\mathrm{F}}+1}\right)\right\}.
$$
Therefore, with such choices, $v$,  $F_{\kappa, \tau}$ and $f_{\kappa, \tau}$ fall into the framework of Approximation Lemma \ref{lem.flat}.
\end{remark}

The next Lemma establishes the first step of the geometric control on the growth of the gradient:

\begin{lemma}\label{lem.firststep}
Under the assumptions of Lemma \ref{lem.flat} there exists $\rho\in \left(0,\frac{1}{2}\right)$ such that
\begin{equation}\label{Aproxcond}
\displaystyle \sup_{B_{\rho}(0)} \frac{\left|u(x)-l_{0} u(x)\right|}{\rho^{1+\beta}}\leq 1.
\end{equation}\label{2.4}
where $\mathfrak{l}_{0} u(x) = u(0)+D u(0)\cdot x$.
\end{lemma}

\begin{proof}
Let $0<\iota\ll 1$ to be chose \textit{a posteriori}. From Lemma \ref{lem.flat} we know that there exists $\delta_\iota>0$, such that whenever
\begin{equation}\label{EqSmallCond}
 \max\left\{\Theta_{\mathrm{F}}(x), \left\|f\right\|_{L^{\infty}(B_{1}(0)\times \R)}\right\}\leq\delta_{\iota},
\end{equation}
then \eqref{eq.flat} holds. For $\rho\in \left(0,\frac{1}{2}\right)$ to be fixed soon and $x\in B_{\rho}(0)$ we compute
\[
   \left|u(x)-\mathfrak{l}_{0} u(x)\right| \leq |u(x)-\mathfrak{h}(x)| + |\mathfrak{h}(x)-\mathfrak{l}_{0} \mathfrak{h}(x)|+ |(D\mathfrak{h}-D u)(0)\cdot x|
\]
so that
\[
\sup_{B_\rho(0)}\left|u(x)-\mathfrak{l}_{0} u(x)\right| \leq \sup_{B_\rho(0)}|\mathfrak{h}(x)-\mathfrak{l}_{0} \mathfrak{h}(x)|+ 2\iota
\]
provided that \eqref{EqSmallCond} there holds.

Now, according to the available regularity theory to homogeneous problem with ``frozen coefficients'' (see, \cite{C89}, \cite{CC95} and \cite{Tru88}) we have
$$
   \displaystyle |\mathfrak{h}(x)-\mathfrak{l}_0 \mathfrak{h}(x)|   \leq  C(N, \lambda, \Lambda)\cdot|x|^{1+\alpha_{\mathrm{F}}} \quad \forall\,\, x \in B_{\frac{1}{2}}(0),
$$
where $C>0$ and $\alpha_{\mathrm{F}} =  \alpha_{\mathrm{F}}(N, \lambda, \Lambda)\in (0, 1]$. Finally,
$$
\begin{array}{rcl}
\displaystyle \sup_{B_\rho(0)}\left|u(x)-\mathfrak{l}_{0} u(x)\right|&\le &  C(N, \lambda, \Lambda)\cdot\rho^{1+\alpha_{\mathrm{F}}}+2\iota\\
&\leq & \rho^{1+\beta}
\end{array}
   .
$$
as long as we make the following universal choices:
\begin{equation}\label{2.5}
  \rho \in \left(0, \min\left\{\frac{1}{2}, \,\left(\frac{3}{4C(N, \lambda, \Lambda)}\right)^{\frac{1}{\alpha_{\mathrm{F}}-\beta}}\right\}\right) \quad \text{and} \quad   \iota \in  \left(0, \frac{1}{8}\rho^{1+\beta}\right)
\end{equation}
Therefore, we obtain \eqref{Aproxcond}, thereby finishing the proof.
\end{proof}

As mentioned before, the previous lemma doesn't allow us to proceed with an iterative procedure. Thus, the following simple consequence will provide the correct estimate.
\begin{corollary}[{\bf $1^{st}$ step of induction}]\label{c3.1}
Suppose that the assumptions of Lemma \ref{lem.firststep} are in force. Then,
$$
\displaystyle \sup_{B_{\rho}(0)}\left|u(x)-u(0)\right|\leq\rho^{1+\beta}+\rho|D u(0)|,
$$
where $\rho$ satisfies \eqref{2.5}.
\end{corollary}

Next, we will obtain the precise control on the influence of the gradient of $u$, we will iterate solutions in suitable dyadic balls. The proof makes use some ideas from \cite[Theorem 3.1]{AdaSRT19}, \cite[ Theorems 1.1 and 1.3]{APR} and \cite[Theorem 1.2]{daS19} and references therein.

\begin{lemma}[{\bf $k^{th}$ step of induction}]\label{lem.dy} Under the assumptions of Lemma \ref{lem.firststep} one has
\begin{equation}\label{3.3}
\displaystyle\sup_{B_{\rho^k}(0)}\left|u(x)-u(0)\right|\leq\rho^{k(1+\beta)}+|D u (0)|\sum_{j=0}^{k-1}\rho^{k+j\beta},
\end{equation}
where $\rho$ satisfies \eqref{2.5}.
\end{lemma}

\begin{proof}
The proof will be via an induction argument.
\begin{enumerate}
  \item The case $k=1$ is precisely the statement of Corollary \ref{c3.1}.
  \item  Suppose now that \eqref{3.3} holds for all the values of $l=1,2,\cdots,k$.
  \item  Our goal is to prove it for $l=k+1$.
\end{enumerate}
Define $v_k: B_1(0) \to \R$ given by
$$
  \displaystyle u_k(x)\defeq \frac{u(\rho^k x)-u(0)}{\mathcal{A}_k}.
$$
Now, by defining
\begin{itemize}
  \item $\mathcal{A}_k \defeq \rho^{k(1+\beta)}+|D u (0)|\sum\limits_{j=0}^{k-1}\rho^{k+j\beta}$
  \item $F_k(x, \mathrm{X}) \defeq \frac{\rho^{2k}}{\mathcal{A}_k} F\left(\rho^k x, \left(\frac{\rho^{2k}}{\mathcal{A}_k}\right)^{-1}\mathrm{X}\right)$;
  \item $f_k(x, s) \defeq \frac{\rho^{k(p+2)} f(\rho^kx, \mathcal{A}_k s)}{\mathcal{A}_k^{p+1}}$;
  \item $\mathfrak{a}_k(x) \defeq \left(\frac{\mathcal{A}_k}{\rho^k}\right)^{q-p}\mathfrak{a}(\rho^k x)$;
  \item $\mathcal{H}_k(x, \xi)\defeq \left(\frac{\rho^{k}}{\mathcal{A}_k}\right)^p\mathcal{H}\left(\rho^kx, \left(\frac{\rho^{k}}{\mathcal{A}_k}\right)^{-1}\xi\right)$.
\end{itemize}
we get, in the viscosity sense,
$$
   \mathcal{H}_k(x, Du_k)F_k(x, D^2 u_k) = f_k(x, u_k) \quad \text{in} \quad B_1(0)
$$
Now, it is easy to check that $\|u_k\|_{L^{\infty}(B_1(0))} \leq 1$ (by induction hypothesis) and
{\small{
\begin{equation}\label{EqIterak}
\text{for every}\,\,\, k \in \mathbb{N}\,\,\,\left\{
\begin{array}{l}
u_k(0) = 0\\
\Theta_{\mathrm{F}_k}(x)\leq \Theta_{\mathrm{F}}(x)\ll1\\
\left\|f_k\right\|_{L^{\infty}(B_{1}(0)\times \R)}\leq \rho^{k[1-\beta(p+1)]} \left\|f\right\|_{L^{\infty}(B_{1}(0)\times \R)}\ll1,
\end{array}
\right.
\end{equation}}}
where we have used the sharp expression \eqref{SharpExp} and the smallness regime (Remark \ref{SmallRegime}). Therefore, $F_k$, $f_k$ and $u_k$ satisfy the assumptions of Approximation Lemma \ref{lem.flat}. Hence, we can apply Corollary \ref{c3.1} to $u_k$ and obtain
$$
\displaystyle \sup_{B_{\rho}(0)} \left|u_k(x)-u_k(0)\right|\leq\rho^{1+\beta}+\rho|D v_k(0)|,
$$
which implies
$$
\displaystyle\sup_{B_{\rho}(0)}\frac{|u(\rho^k x)-u(0)|}{\mathcal{A}_k}\leq \rho^{1+\beta}+\frac{\rho^{k+1}|D u(0)|}{\mathcal{A}_k}.
$$
Finally, by scaling back to the unit domain, we get
$$
  \displaystyle\sup_{B_{\rho^{k+1}}(0)}|u(x)-u(0)| \leq  \rho^{(k+1)(1+\beta)}+ |D u(0)|\sum\limits_{j=0}^{k}\rho^{k+1+j\beta},
$$
thereby obtaining the $(k+1)-$step of induction.
\end{proof}

The next result leads to a sharp regularity estimate inside the singular zone

\begin{lemma}\label{l3.3}
Suppose that the assumptions of Lemma \ref{lem.flat} are in force. Then, there exists a universal constant $\mathrm{M}_0>1$ such that, for $\rho$ as in the conclusion of that Lemma,
$$
\displaystyle \sup_{B_{r}(0)}|u(x)-u(0)|\leq \mathrm{M}_0 \cdot r^{1+\beta}\left(1+|D u(0)|r^{-\beta}\right),\,\,\forall r\in(0,\rho).
$$
\end{lemma}

\begin{proof}
Firstly, fix any $r\in(0,\rho)$ and choose $k\in\mathbb{N}$ the smallest integer such that $\rho^{k+1}<r\leq\rho^{k}$. By using Lemma \ref{lem.dy}, we estimate
\begin{align*}
\sup_{B_r(0)}\frac{|u(x)-u(0)|}{r^{1+\beta}} & \leq \frac{1}{\rho^{1+\beta}} \sup_{B_{\rho^k}(0)}\frac{|u(x)-u(0)|}{\rho^{k(1+\beta)}} \\
										   & \displaystyle \leq \frac{1}{\rho^{1+\beta}} \left(1+|Du(0)|\rho^{-k(1+\beta)}\sum_{j=0}^{k-1}\rho^{k+j\beta}\right)\\
										   & \leq \frac{1}{\rho^{1+\beta}} \left(1+|Du(0)|\rho^{-k\beta}\sum_{j=0}^{k-1}\rho^{j\beta}\right)\\
										   & \leq \frac{1}{\rho^{1+\beta}} \left(1+|Du(0)|\rho^{-k\beta}\frac{1}{1-\rho^\beta}\right)\\
   										   & \leq \mathrm{M}_0\cdot (1+|Du(0)|r^{-\beta}),\\
\end{align*}
where $\mathrm{M}_0 \defeq \frac{1}{\rho^{1+\beta}(1-\rho^\beta)}$, thereby concluding the proof.
\end{proof}

\section{Proofs of the main results}\label{sec.ProofMRes}

Now, we can give the proof of first main result of this manuscript:

\begin{proof}[{\bf Proof of Theorem \ref{main1}}]
Without loss of generality, we may assume that $x_0=0$. Notice that the degenerate ellipticity of the operator naturally leads us to separate the study into two different regimes depending on whether $|Du(0)|$ is ``sufficiently small'' or not.

\vspace{0.3cm}

\begin{enumerate}
  \item If $0 \in \mathcal{S}_{r, \beta}(u, \Omega)$
\vspace{0.1cm}

By using Lemma \ref{lem.dy} we estimate
\begin{align*}
\sup_{B_r(0)}\left|u(x)-\mathfrak{l}_{0} u(x)\right| & \leq \sup_{B_r(0)}|u(x)-u(0)|+|D u(0)|r \\
												 &  \leq \mathrm{M}_0 \cdot r^{1+\beta}\left(1+|D u(0)|r^{-\beta}\right)+r^{1+ \beta}\\
												 & \leq  3\mathrm{M}_0 \cdot r^{1+\beta}\\
\end{align*}
as desired in this case.

\vspace{0.3cm}
  \item If $0 \notin \mathcal{S}_{r, \beta}(u, \Omega)$ i.e. $r^\beta< |D u(0)|\leq L$
\vspace{0.1cm}

In this case, let us define $r_0 \defeq |D u(0)|^{\frac{1}{\beta}}$ and
$$
u_{r_0}(x) \defeq \frac{u(r_0x)-u(0)}{r_0^{1+\beta}}.
$$
Hence, we are allowed to apply Lemma \ref{l3.3} and conclude that
\begin{equation}\label{EstSmallGrad}
\displaystyle \sup_{B_{r_0}(0)} |u(r_0x)-u(0)| \leq 2\mathrm{M}_0\cdot r_0^{1+\beta}.
\end{equation}
Now, notice that $u_{r_0}$ fulfills in the viscosity sense
$$
\mathcal{H}_{r_0}(x, D u_{r_0})F_{r_0}(x, D^2 u_{r_0}) = f_{r_0}(x, u_{r_0}) \quad \text{in} \quad B_1(0),
$$
where
$$
\left\{
\begin{array}{rcl}
F_{r_0}(x, \mathrm{X}) & \defeq & r_0^{1-\beta}F\left(r_0x, r_0^{-(1-\beta)}\mathrm{X}\right) \\
f_{r_0}(x, s)  & \defeq & r_0^{1-\beta(p+1)}f(r_0x, r_0^{1+\beta}s)\\
\mathcal{H}_{r_0}(x, \xi ) & \defeq & r_0^{-p\beta}\mathcal{H}\left(r_0x, r_0^{\beta}\xi\right)\\
\mathfrak{a}_{r_0}(x) & \defeq & r_0^{(q-p)\beta}\mathfrak{a}(r_0x)
\end{array}
\right.
$$
and
\begin{equation}\label{Eq3.2}
\displaystyle u_{r_0}(0) = 0, \,\,\,|D u_{r_0}(0)| = 1 \quad \text{and} \quad \|f_{r_0}\|_{L^{\infty}(B_1(0) \times \R)} \le 1.
\end{equation}
Moreover, \eqref{EstSmallGrad} assures us that $u_{r_0}$ is uniformly bounded in the $L^{\infty}-$topology. From Theorem \ref{GradThm} it follows (using \eqref{Eq3.2}) that
$$
\|u_{r_0}\|_{C^{1,\gamma}(B_{1/2}(0))}\leq C \quad (\text{for a universal constant}).
$$
Such an estimate and one more time \eqref{Eq3.2}, allow us to choose a universal radius $0<\rho_0\ll 1$ such that
$$
\mathfrak{c}_0 \leq |D u_{r_0}(x)|\leq \mathfrak{c}_0^{-1} \,\,\,\forall \,\,x \in B_{\rho_0}(0) \,\,\,\text{and}\,\,\, \mathfrak{c}_0 \in (0, 1) \,\,\,\text{fixed}.
$$

Particularly, we obtain (in the viscosity sense)
\[
 F_{r_0}(x, D^2u_{r_0}) =\tilde{f}_{r_0}(x, u_{r_0})\defeq \frac{f(x, u_{r_0})}{\mathcal{H}_{r_0}(x, |Du_{r_0}|)} \quad \text{in} \quad B_{\rho_0}(0),
\]

The previous statement says $\tilde{f}_{r_0}$ is (universally) bounded in $B_{\rho_0}(0)$ and we get the result from classical estimates (see, \cite{C89}, \cite[Section 8.2]{CC95} and \cite{Tru88}) since the equation becomes uniformly elliptic:
$$
  \mathcal{M}_{\lambda, \Lambda}^-(D^2 u_{r_0})\leq C_0\left(p, q, \mathfrak{c}_0, L^{-1}_1, \|f\|_{L^{\infty}(\Omega \times \R)}, \|\mathfrak{a}\|_{L^{\infty}(\Omega)}\right)
$$
and
$$
\mathcal{M}_{\lambda, \Lambda}^+(D^2 u_{r_0})\geq -C_0\left(p, q, \mathfrak{c}_0, L^{-1}_1, \|f\|_{L^{\infty}(\Omega \times \R)}, \|\mathfrak{a}\|_{L^{\infty}(\Omega)}\right).
$$

 Therefore, $u_{r_0}\in C_{\text{loc}}^{1, \alpha}(B_{\rho_0}(0))$ for every $\alpha \in (0, \alpha_{\mathrm{F}})$. As a consequence, we have
 \begin{equation}\label{EqEstUnifElliOper}
  \displaystyle \sup_{B_r(0)}\left|u_{r_0}(x)-\mathfrak{l}_{0} u_{r_0}(x)\right|\leq C \cdot r^{1+\alpha}, \,\,\,\forall\,\,r \in \left(0, \frac{\rho_0}{2}\right),
 \end{equation}
 which, one translates in terms of $u$ as follows
 $$
 \displaystyle \sup_{B_r(0)}\left|\frac{u(r_0x)-u(0)}{r_0^{1+\beta}}-r_0^{-\beta}Du(0)\cdot x\right|\leq C \cdot r^{1+\alpha}.
 $$
 Next, by doing $\beta=\alpha$ in the above estimate (see, \eqref{SharpExp}), we conclude
  $$
 \displaystyle \sup_{B_r(0)}\left|u(x)-\mathfrak{l}_{0} u(x)\right|\leq C \cdot r^{1+\beta}, \,\,\,\forall\,\,r \in \left(0, \frac{\rho_0r_0}{2}\right),
 $$

 Finally, for $ r \in \left[\frac{\rho_0r_0}{2}, r_0\right)$, we obtain
 $$
 \begin{array}{rcl}
 \displaystyle \sup_{B_r(0)}\left|u(x)-\mathfrak{l}_{0} u(x)\right| & \leq &\displaystyle \sup_{B_{r_0}(0)}\left|u(x)-\mathfrak{l}_{0} u(x)\right|\\
 &\leq &\displaystyle \sup_{B_{r_0}(0)}\left|u(x)-u(0)\right| +|Du(0)|r_0\\
 &\leq & (2\mathrm{M}_0+1)\cdot r_0^{1+\beta}\\
 &\leq & 3\mathrm{M}_0 \left(\frac{2}{\rho_0}\right)^{1+\beta}\cdot r^{1+\beta}.
 \end{array}
   $$

In conclusion, from characterization of Dini-Campanato spaces in \cite{Kov99} we deduce that $u$ is $C^{1, \beta}$ at $x_0 = 0$. Furthermore, a standard covering argument yields the corresponding estimate in any subset $\Omega^{\prime} \Subset \Omega$, which completes the proof of Theorem.

\end{enumerate}
\end{proof}

\begin{proof}[{\bf Proof of Corollary \ref{Cormain1}}]
It follows immediately from Theorem \ref{main1}, since solutions to the homogeneous problem (with ``frozen coefficients'') for such classes of operators are $C_{\text{loc}}^{1,1}(\Omega)$, i.e., $\alpha_{\mathrm{F}}=1$ (see, Section \ref{Examples}). For this reason, we can choose $\beta = \frac{1}{p+1} \in (0, 1)$ in the sentences \eqref{2.5}, \eqref{EqIterak} and \eqref{EqEstUnifElliOper}.
\end{proof}

As mentioned before, with the aid of Theorem \ref{main1} we can prove the growth control on the gradient stated in Corollary\ref{Cormain2}, thus obtaining a finer gradient control to solutions of \eqref{1.1} at interior points.

\begin{proof}[{\bf Proof of Corollary \ref{Cormain2}}]  Let $x_0 \in  \Omega^{\prime} \Subset \Omega$ be an interior point. Now, we define the scaled auxiliary function $u_{r, x_0}: B_1(0) \to \R$ by:
$$
  u_{r, x_0}(x) \defeq \frac{u(x_0+rx)-u(x_0)-rx \cdot Du(x_0)}{r^{1+\beta}}.
$$
Now, observe that $u_{r, x_0}$ fulfills in the viscosity sense
$$
  \mathcal{H}_{r, x_0}(x, D u_{r, x_0}+\xi_{r, x_0})F_{r, x_0}(x, D^2 u_{r, x_0}) = f_{r, x_0}(x, s) \quad \text{in} \quad B_1(0),
$$
where
$$
\left\{
\begin{array}{rcl}
  F_{r, x_0}(x, X) & \defeq & r^{1-\beta}F\left(x_0+rx, \frac{1}{r^{1-\beta}}X\right) \\
  f_{r, x_0}(x, s)  & \defeq & r^{1-(p+1)\beta}f(x_0 + r x, r^{1+\beta}s)\\
  \mathcal{H}_{r, x_0}(x, \xi ) & \defeq & r^{-\beta p}\mathcal{H}(x_0+rx, r^{\beta}\xi)\\
  \mathfrak{a}_{r, x_0}(x) & \defeq & r^{(q-p)\beta}\mathfrak{a}(x_0+rx)\\
  \xi_{r, x_0} & \defeq & r^{-\beta}Du(x_0).
\end{array}
\right.
$$
From Theorem \ref{main1} we get that
$$
\|u_{r, x_0}\|_{L^{\infty}\left(B_{\frac{1}{4}}\right)} \leq C \cdot \left(\|u\|_{L^{\infty}(B_1(0))} + 1+ \|f\|_{L^{\infty}(B_1(0)\times \R)}^{\frac{1}{p+1}}\right)
$$
Finally, by invoking the gradient estimate (Theorem \ref{GradThm}) we obtain that
{\small{
$$
\begin{array}{rcl}
  \displaystyle   \sup_{B_{\frac{r}{8}}(x_0)} \frac{|D u(x)-Du(x_0)|}{r^{\beta}} & = & \displaystyle \sup_{B_{\frac{1}{8}}(x_0)} |D u_{r, x_0}(y)| \\
   & \leq  & C \cdot \left(\|u_{r, x_0}\|_{L^{\infty}\left(B_{\frac{1}{4}}(0)\right)} + 1+ \|f_{r, x_0}\|_{L^{\infty}\left(B_{\frac{1}{4}}(0)\times \R\right)}^{\frac{1}{p+1}}\right) \\
   & \le & C_0 \cdot \left(\|u\|_{L^{\infty}(B_1(0))} + 1 +\|f\|_{L^{\infty}(B_1(0)\times \R)}^{\frac{1}{p+1}}\right),
\end{array}
$$}}
thereby finishing the proof.
\end{proof}

Before proving our last main result, let us present a comparison tool. The proof holds the same ideas as ones in \cite[Theorem 1.1]{BD0} and \cite[Theorem 2.1]{BD2}. For this reason, we will omit the proof here.

\begin{lemma}[{\bf Comparison Principle}]\label{comparison principle}
	Assume that assumptions (A0)-(A2) there hold. Let $f \in C^0(\bar{\Omega})$ and $h$ be a continuous increasing function satisfying $h(0) = 0$. Suppose $u_1$ and $u_2$ are respectively a viscosity supersolution and 	subsolution of
	$$
	\mathcal{H}(x, Dw) F(x, D^2w) = h(w) + f(x) \quad\text{in} \quad \Omega.
	$$
If $u_1 \geq u_2$ on $\partial \Omega$, then $u_1 \geq u_2$ in $\Omega$.

Furthermore, if $h$ is nondecreasing (in particular if $h \equiv 0$), the result holds if $u_1$ is a strict supersolution or vice versa if $u_2$ is a strict subsolution.
\end{lemma}

Finally, we are in a position to prove the non-degeneracy property.

\begin{proof}[{\bf Proof of Theorem \ref{main3}}]

Firstly, let us introduce the comparison function:
$$
   \Theta(x) \defeq \mathfrak{c}\cdot |x|^{\frac{p+2}{p+1}},
$$
where the constant $\mathfrak{c}>0$ will be chosen in such a way that
$$
\mathcal{H}(x, D\Theta)F(x, D^2 \Theta) < f(x)\quad \text{in} \quad B_R(0) \Subset \Omega
$$
Note that
\begin{eqnarray*}
	D \Theta(x) &=& \mathfrak{c}\cdot \left(\frac{p+2}{p+1}\right) |x|^{-\frac{1}{p+1}}x\\
	D^2 \Theta(x)& =& \mathfrak{c} \cdot \left(\frac{p+2}{p+1}\right)\left(\textrm{Id}_N - \frac{p}{p+1} |x|^{-p} x \otimes x \right)
\end{eqnarray*}
Direct computation shows that
$$
	\mathcal{M}^{+}_{\lambda, \Lambda}(D^2 \Theta(x)) \le \mathfrak{c} \cdot \left(\frac{p+2}{p+1}\right)\left( \frac{1}{p+1} \lambda + (N-1) \Lambda \right) |x|^{-\frac{p}{p+1}}
$$
Thus, from (A1), \eqref{1.2} and \eqref{N-HDeg} assumptions we have
$$
\begin{array}{rcl}
  \mathcal{H}(x, D\Theta)F(x, D^2 \Theta) & \leq & L_1 \cdot \left[|D\Theta|^p + \mathfrak{a}(x) |D\Theta|^q \right] \mathcal{M}^{+}_{\lambda, \Lambda}(D^2 \Theta(x)) \\
   & \le & \Xi_1\cdot \Xi_2,
\end{array}
$$
where
$$
\Xi_1 \defeq \left[\left(\frac{p+2}{p+1}\right)^{p+1} \mathfrak{c}^{p+1} + \|\mathfrak{a}\|_{L^{\infty}(\Omega)} \left(\frac{p+2}{p+1}\right)^{q+1}\mathfrak{c}^{q+1}\left(\frac{\diam(\Omega)}{2}\right)^{\frac{q-p}{p+1}}\right],
$$
and
$$
\Xi_2 \defeq L_1\cdot \left( \frac{1}{p+1} \lambda + (N-1) \Lambda\right).
$$
At this point, consider the analytical function $\mathfrak{g}: [0, \infty) \to \R$ given by
$$
   \mathfrak{g}(t) \defeq \Xi_2\cdot t^{p+1}\left[\left(\frac{p+2}{p+1}\right)^{p+1}+\Xi_3\cdot t^{q-p}\right]-\mathfrak{m},
$$
where
$$
    \displaystyle \Xi_3 \defeq \|\mathfrak{a}\|_{L^{\infty}(\Omega)} \left(\frac{p+2}{p+1}\right)^{q+1}\left(\frac{\diam(\Omega)}{2}\right)^{\frac{q-p}{p+1}} \quad \text{and} \quad \mathfrak{m} \defeq \inf_{\Omega} f(x).
$$
Now, let us label $\mathrm{T}_0$ its smallest root, which there exists thanks to the assumption $\displaystyle \mathfrak{m} \defeq \inf_{\Omega} f(x)>0$. Therefore, we are able to choose a $\mathfrak{c}=\mathfrak{c}(\mathfrak{m},\|\mathfrak{a}\|_{L^{\infty}(\Omega)}, L_1, N, \lambda, \Lambda, p,q, \Omega) \in (0, \mathrm{T}_0)$ such that
$$
 \mathcal{H}(x, D\Theta)F(x, D^2 \Theta) < f(x) \qquad \text{point-wisely}.
$$

In the sequel, for $x_0 \in \Omega^{\prime} \Subset \Omega$ let us define the scaled function:
 $$
    u_{r, x_0}(x) \defeq \frac{u(x_0+rx)-u(x_0)+ \varepsilon}{r^{\frac{p+2}{p+1}}} \quad \text{for} \quad x \in B_1(0).
 $$
Now, observe that $u_{r, x_0}$ fulfills in the viscosity sense
$$
  \mathcal{H}_{r, x_0}(x, D u_{r, x_0})F_{r, x_0}(x, D^2 u_{r, x_0}) = f_{r, x_0}(x) \quad \text{in} \quad B_1(0),
$$
where
$$
\left\{
\begin{array}{rcl}
  F_{r, x_0}(x, \mathrm{X}) & \defeq & r^{\frac{p}{p+1}}F\left(x_0+rx, r^{-\frac{p}{p+1}}\mathrm{X}\right) \\
  f_{r, x_0}(x)  & \defeq & f(x_0 + r x)\\
  \mathcal{H}_{r, x_0}(x, \xi ) & \defeq & r^{-\frac{p}{p+1}}\mathcal{H}\left(x_0+rx, r^{\frac{1}{p+1}}\xi\right)\\
  \mathfrak{a}_{r, x_0}(x) & \defeq & r^{\frac{q-p}{p+1}}\mathfrak{a}(x_0+rx).
\end{array}
\right.
$$

Finally, if $u_{r, x_0} \leq \Theta$ on the whole boundary of $B_1(0)$, then the Comparison Principle (Lemma \ref{comparison principle}), would imply that
$$
   u_{r, x_0}(x) \leq \Theta(x) \quad \mbox{in} \quad B_1(0),
$$
which contradicts the assumption that $u_{r, x_0}(0)> 0$. Therefore, there exists a point $z \in \partial B_1(0)$ such that
$$
      u_{r, x_0}(z) > \Theta(z) = \mathfrak{c}(\mathfrak{m},\|\mathfrak{a}\|_{L^{\infty}(\Omega)}, L_1, N, \lambda, \Lambda, p,q, \Omega).
$$
By scaling back and letting $\varepsilon \to 0$ we finish the proof of the Theorem.
\end{proof}


\section{Connections with geometric free boundary problems}\label{sec.Applic}

In the sequel, we will present scenarios where our results also take place.

\subsection{Dead-core type problems}
The main purpose of this section is to study the dead-core problem for fully nonlinear models with non-homogeneous degeneracy, whose source term presents an absorption term:
\begin{equation}\label{Maineq}
	\mathcal{H}(x,D u).F(x, D^2 u) = f(x)\cdot u^{\mu}\chi_{\{u>0\}} \quad \textrm{in} \quad \Omega,
\end{equation}
where $0 < \mu < p+1$ is the order of reaction and $f$ is the Thiele modulus, which is bounded away from zero and infinity. We shall establish an improved regularity estimate for non-negative solutions of \eqref{Maineq} along their touching ground boundary $\partial\{u>0\}$ in contrast with Theorem \ref{main1}. This is an important piece of information in several free boundary problems (cf. \cite{daSLR20}, \cite{daSO19}, \cite{daSOS18}, \cite{daSRS19I}, \cite{daSS18} and \cite{Tei16} for more explanations)

Now, let us comment on the existence of a viscosity solution to the Dirichlet problem \eqref{Maineq}. Such an existence result follows by an application of Perron's method since a version of the Comparison Principle (Lemma \ref{comparison principle}) is available. In fact, let us consider functions $u^{\sharp}$ and $u_{\flat}$ viscosity solutions to the following boundary value problems:
$$\left\{
\begin{array}{rcccc}
\mathcal{H}(x, D u^{\sharp})F(x, D^2 u^{\sharp}) & = & 0 & \mbox{in} & \Omega, \\
u^{\sharp}(x) & = & g(x) &  \mbox{on} & \partial\Omega.\\
\end{array}
\right.
$$
and
$$
\left\{
\begin{array}{rllcc}
\mathcal{H}(x, D u_{\flat})F(x, D^2 u_{\flat}) &=& \|f\|_{L^\infty(\Omega)}\|g\|_{L^\infty(\partial\Omega)}^{\mu} & \mbox{in} & \Omega, \\
u_{\flat}(x) &=& g(x) &  \mbox{on} & \partial\Omega.\\
\end{array}
\right.
$$
The existence of such solutions follows via standard arguments. Moreover, notice that $u^{\sharp}$ and $u_{\flat}$ are, respectively, super-solution and sub-solution to \eqref{Maineq} (with non-negative continuous boundary datum $g$). Consequently, by Comparison Principle, Lemma \ref{comparison principle}, it is possible, under a direct application of Perron's method, to obtain the existence of a viscosity solution in $C^0(\overline{\Omega})$ to \eqref{Maineq}. Precisely, we have the following result:

\begin{theorem}[{\bf Existence and uniqueness}]\label{ThmExist} Let $f^{\ast} \in C^0([0, \infty)) $ be a bounded, increasing real function with $f^{\ast}(0)=0$. Suppose that the assumptions (A0)-(A1), \eqref{1.2} and \eqref{1.3} are in force. Suppose further that there exist a viscosity sub-solution $u_{\flat} \in C^0(\overline{\Omega}) \cap C^{0, 1}(\Omega)$ and a viscosity super-solution $u^{\sharp} \in C^0(\overline{\Omega}) \cap C^{0, 1}(\Omega)$ to
\begin{equation}\label{EqExistDC}
  \mathcal{H}(x, Du)F(x, D^2 u) = f^{\ast}(u) \quad \text{in} \quad \Omega,
\end{equation}
satisfying
$u_{\flat} = u^{\sharp} = g \in C^0(\partial \Omega)$. Define the class of functions
$$
     \mathrm{S}_{g}(\Omega) \defeq \left\{ v \in C^0(\overline{\Omega}) \;\middle|\; \begin{array}{c}
 v \text{ is a viscosity super-solution to } \\
\eqref{EqExistDC} \text{ such that } u_{\flat} \le v \le u^{\sharp}\\
\text{ and } v = g \text{ on } \partial \Omega
\end{array}
\right\}.
$$
Then,
$$
   	u(x) \defeq \inf_{\mathrm{S}_{g}(\Omega)} v(x), \,\,\,\, \mbox{for} \,\, x \in \overline{\Omega}
$$
is the unique continuous, up to the boundary, viscosity solution to
$$
\left\{
\begin{array}{rclcl}
  \mathcal{H}(x, Du)F(x, D^2u) & = & f^{\ast}(u) & \mbox{in} & \Omega \\
  u(x) & = & g(x) & \mbox{on} & \partial \Omega.
\end{array}
\right.
$$
\end{theorem}

Next result regards the first step of a sharp geometric decay, which is a powerful device in nonlinear (geometric) regularity theory and plays a pivotal role in our approach.

\begin{lemma}[{\bf Flatness improvement regime}]\label{FlatLemma} Suppose that the assumptions (A0)-(A1), \eqref{1.2} and \eqref{1.3} are in force. Given $0<\eta<1$, there exists a $\delta = \delta(N, \lambda, \Lambda, p, \eta)>0$ such that if $\phi$ satisfies $0 \le \phi \le 1$, $\phi(0)=0$ and
\begin{equation} \label{1}
	\mathcal{H}(x,D \phi).F(x,D^2 \phi) = f(x)\cdot (\phi^+)^{\mu},
\end{equation}
in the viscosity sense in $B_1(0)$, with $\|f\|_{L^{\infty}(B_1(0))} \le \delta$. Then,
\begin{equation}
   \sup_{B_{1/2}(0) } \phi \le 1-\eta.
\end{equation}
\end{lemma}

\begin{proof} Suppose for the sake of contradiction that the thesis of Lemma fails to hold. This means that for some $\eta_0\in (0, 1)$ and for each $k \in \mathbb{N}$, we could find sequences $\{ \phi_k \}_{k}$, $\{F_k\}_{k}, \{\mathfrak{a}_k\}_k$ and $\{f_k\}_k$ satisfying $0 \le \phi_k \le 1$, $\phi_k(0)=0$, $\|f_k\|_{L^{\infty}(B_1(0))} \le \frac{1}{k}$ and
$$
	\mathcal{H}_k(x, D \phi_k)F_k(x,D^2 \phi_k) = f_k(x)\cdot (\phi^+_k)^{\mu}  \quad \textrm{in} \quad B_1(0),
$$
in the viscosity sense. However,
  \begin{equation}\label{Eqcont}
    \sup_{B_{1/2}(0)} \phi_k > 1-\eta_0
  \end{equation}
for all $k \ge 1$. Notice that $\|f_k (\phi^+_k)^{\mu}\|_{L^{\infty}(B_1(0))} \le \frac{1}{k}$. Moreover, by H\"{o}lder regularity of solutions (see, \cite[Proposition 3.3]{DeF20}), up to a subsequence, $\phi_k \to \phi_{\infty}$ local uniformly in $B_{2/3}(0)$. Notice that such a uniform limit verifies
$$
  0 \le \phi_{\infty}(x) \le 1  \quad \text{and} \quad \phi_{\infty}(0) = 0
$$
Moreover, by arguing as \cite[Lemma 4.1]{DeF20} and making use of stability results for viscosity solutions (cf. \cite[Corollary 2.7 and Remark 2.8]{BD2}), the limiting function $\phi_{\infty}$ satisfies in the viscosity sense
$$
	\mathcal{M}^{-}_{\lambda, \Lambda}(D^2 \phi_{\infty}) \le 0 \le \mathcal{M}^{+}_{\lambda,\Lambda}(D^2 \phi_{\infty}) \quad \text{in} \quad B_{2/3}(0).
$$

Finally, we conclude that $\phi_{\infty} \equiv 0$ in $B_{2/3}(0)$ via the Strong Maximum Principle (see, \cite[Proposition 4.9]{CC95}), which yields a contradiction with \eqref{Eqcont} by choosing $k \gg 1$ large enough, thereby finishing the proof.
\end{proof}

In the sequel, by applying Lemma \ref{FlatLemma} recursively in dyadic balls $B_{\frac{1}{2^k}}(0)$ with $\eta \defeq 1-\left(\frac{1}{2}\right)^{^{\frac{p+2}{p+1-\mu}}}$, we are able to establish improved regularity estimates along touching ground points. The proof follows the same lines as \cite[Theorem 1.2]{daSLR20} and \cite[Theorem 2]{Tei16}. For this reason, we will omit it.

\begin{theorem}[{\bf Improved regularity along free boundary}]\label{IRThm}
Let $u$ be a nonnegative and bounded viscosity solution to \eqref{Maineq} and consider $z_0 \in \partial\{u > 0\} \cap \Omega^{\prime}$ a free boundary point with $\Omega^{\prime} \Subset \Omega$. Then for $r_0 \ll  \min\left\{1, \frac{\dist(\Omega^{\prime}, \partial \Omega)}{2}\right\}$ and any $x \in B_{r_0}(z_0) \cap \{u > 0\}$ there holds
$$
   u(x)\leq \mathfrak{C}^{\sharp}\cdot \max\left\{1, \|u\|_{L^\infty(\Omega)} \right\}.|x-z_0|^{\frac{p+2}{p+1-\mu}},
$$
where $\mathfrak{C}^{\sharp}>0$ depends only on $\displaystyle N, \lambda, \Lambda, \mu, \|f\|_{L^{\infty}(\Omega)}$ and $\dist(\Omega^{\prime}, \partial \Omega)$.
\end{theorem}

 Thanks to Theorem \ref{IRThm} we are able to access better regularity estimates (at free boundary points) than those previously available. As a matter of fact, in such a result, the modulus of continuity improves upon the expected $C^{1, \beta}$ regularity coming from Theorem \ref{main1}. By way of comparison, if $F$ is under the assumptions of the Corollary \ref{Cormain1}, Theorem \ref{IRThm} address a sharp/improved modulus of continuity, at free boundary points, i.e.
$$
   \kappa(p, \mu) = \frac{p+2}{p+1-\mu} > 1+\frac{1}{p+1},  \quad \,\,\,\,(\text{sharp and improved  exponent}).
$$

In contrast with Corollary \ref{Cormain2} we also find the sharp/improved rate of gradient's decay at interior free boundary points (cf. \cite[Theorem 1.4]{daSLR20}).

\begin{corollary}[{\bf Sharp gradient's decay}]\label{IRresult2} Let $u$ be a bounded non-negative viscosity solution to \eqref{Maineq}. Then, for any point $z \in \partial \{u > 0\} \cap \Omega^{\prime}$ for $\Omega^{\prime} \Subset \Omega$, there exists a universal constant $C>0$ such that
$$
   \displaystyle \sup_{B_r(z)} |D u(x)| \leq  C\cdot r^{\frac{1+\mu}{p+1-\mu}} \quad \text{for all} \quad  0<r \ll \min\left\{1, \frac{\dist(\Omega^{\prime}, \partial \Omega)}{2}\right\}.
$$
\end{corollary}

Now, we show that the maximum of a solution within a blab of radius $0<r\lll 1$ does growth precisely as $r^{\frac{p+2}{p+1-\mu}}$.

\begin{theorem}[{\bf Non-degeneracy}]\label{LGR} Let $u$ be a nonnegative, bounded viscosity solution to \eqref{Maineq} in $B_1(0)$ with $f(x) \geq \mathfrak{m}>0$ and let $x_0 \in \overline{\{u >0\}} \cap B_{\frac{1}{2}}(0)$ be a point in the closure of the non-coincidence set. Then for any $0<r<\frac{1}{2}$, there holds
$$
   \displaystyle \sup_{\partial B_r(x_0)} \,u(x) \geq \mathfrak{C}\cdot r^{\frac{p+2}{p+1-\mu}},
$$
where $\displaystyle \mathfrak{C} = \mathfrak{C}(\mathfrak{m},\|\mathfrak{a}\|_{L^{\infty}(\Omega)}, L_1, N, \lambda, \Lambda, p, q, \mu, \Omega)>0$.
\end{theorem}

\begin{proof}
 Firstly, for $x_0 \in \{u>0\} \cap \Omega^{\prime}$ let us define the scaled function
$$
   u_r(x) \defeq \frac{u(x_0+rx)}{r^{\frac{p+2}{p+1-\mu}}} \quad  \text{for} \quad  x \in B_1(0).
$$
Straightforward calculus shows that
$$
\mathcal{H}_r(x, D u_r)F_r(x,D^2 u_r)\ge  f_r(x)\cdot (u_r)_{+}^{\mu}(x)   \quad \text{in} \quad B_1(0)
$$
in the viscosity sense, where
$$
\left\{
\begin{array}{rcl}
  F_{r}(x, \mathrm{X}) & \defeq & r^{\frac{p-2 \mu}{p+1-\mu}}F\left(x_0+rx, r^{-\frac{p-2 \mu}{p+1-\mu}}\mathrm{X}\right) \\
  \mathcal{H}_{r}(x, \xi ) & \defeq & r^{-\frac{p(\mu+1)}{p+1-\mu}}\mathcal{H}\left(x_0+rx, r^{\frac{\mu +1}{p+1-\mu}}\xi\right)\\
  \mathfrak{a}_{r}(x) & \defeq & r^{\frac{(q-p)(\mu+1)}{p+1-\mu}}\mathfrak{a}(x_0+rx)\\
  f_r (x) & \defeq & f(x_0 +rx).
\end{array}
\right.
$$
Now, let us introduce the auxiliary function:
$$
	\displaystyle \Xi (x) \defeq \mathfrak{C}\cdot |x|^{\frac{p+2}{p+1-\mu}},
$$
where the constant $\mathfrak{C}>0$ will be chosen in such a way that
$$
	\mathcal{H}(x, D \Xi)F(x, D^2 \Xi) \leq f(x)\cdot \Xi^{\mu}(x) \quad \textrm{in} \quad B_1(0)
$$
At this point, the conclusion can easily be obtained following the lines of the proof of Theorem \ref{main3} by using the Comparison Principle (Lemma \ref{comparison principle}).
\end{proof}

As a consequence of Theorems \ref{IRThm} and \ref{LGR}, we conclude that the non-coincidence set $\{u>0\}$ has uniform positive density and it is a porous set. Particularly, such a result implies that touching ground boundary cannot develop cusp points.

\begin{corollary}\label{Cor4.2}
Let $u$ be a nonnegative, bounded viscosity solution to \eqref{Maineq} in $B_1(0)$ and $x_0 \in \partial \{u>0\} \cap B_{1/2}(0)$ be a free boundary point. Then for any $0 < r < 1/2$,
$$
	\frac{\mathcal{L}^N\left( B_{r}(x_0)\cap \{u>0\}\right)}{\mathcal{L}^N\left( B_{r}(x_0)\right)} \ge \theta,
$$
for a constant $\theta=\theta(N,\lambda,\Lambda,p,q,\|f\|_{\infty}, \mu) >0$. There further exists a universal constant $\epsilon= \epsilon(N,\lambda,\Lambda, \mu) >0$ such that
$$
	\mathscr{H}^{N-\epsilon} \left(\partial \{u>0\} \cap B_{\frac{1}{2}}(0)\right) < \infty.
$$
Particularly, the free boundary has zero Lebesgue measure (see, \cite{Zaj}).
\end{corollary}
\begin{proof}
The proof follows as the one in \cite[Corollaries 5.2 and 5.4]{daSLR20}.
\end{proof}

An interesting open issue consists in knowing whether the result of Corollary \ref{Cor4.2} could be proved in terms of the intrinsic Hausdorff measures introduced in \cite{DeFM19I}. In future work, we intend to address this question.

\subsection{Obstacle type problems}

The main purpose of this section is to comment on sharp regularity for solutions of obstacle type problems driven by our class of operators. More precisely, we consider viscosity solutions of
\begin{equation} \label{Eq1}
\left\{
\begin{array}{rclcl}
 \mathcal{H}(x,D u).F(x,D^2 u) & = & f(x)\cdot \chi_{\{u>\phi\}} & \text{in} & B_{1}(0) \\
  u(x) & \ge & \phi(x) & \text{on} &  B_{1}(0)\\
  u(x) &=& g(x) & \text{on} & \partial B_1(0),
\end{array}
\right.
\end{equation}
with $\phi \in C^{1,1}(B_1(0))$, $f \in L^{\infty}(B_1) \cap C^0(B_1(0))$ and $g$ is a continuous boundary datum.  We prove that they are $C^{1, \frac{1}{p+1}}(B_{1/2}(0))$.  In this context, obstacle type problems have been studied in recent years in \cite{DSVII} (see also \cite{DSVI} for another class of obstacle problems of degenerate type), where were proved $C^{1, \alpha}$ estimates for a class of degenerate fully nonlinear operators as follows
$$
\left\{
\begin{array}{rclcl}
 |D u|^pF(x,D^2 u) & = & f(x)\cdot \chi_{\{u>\phi\}} & \text{in} & B_{1}(0) \\
  u(x) & \ge & \phi(x) & \text{on} &  B_{1}(0)\\
  u(x) &=& g(x) & \text{on} & \partial B_1(0),
\end{array}
\right.
$$

It is worth noting that existence/uniqueness assertions of a viscosity solution to Dirichlet problem \eqref{Eq1} follow by Perron's method combined with penalization techniques (see, \cite[Theorem 1.1]{DSVI} and \cite[Appendix]{DSVII}).

Finally, by making use of the ideas in \cite[Theorem 1.3]{DSVII}, we are in a position to state the following result:

\begin{theorem}[{\bf Regularity along free boundary points}]\label{T1}
Suppose that the assumption (A0)-(A2), \eqref{1.2} and \eqref{1.3} are in force for a convex or concave operator $F$. Let $u$ be a bounded viscosity solution to \eqref{Eq1} with obstacle $\phi \in C^{1,1}(B_1(0))$ and $f \in L^{\infty}(B_1(0))\cap C^0(B_1(0))$. Then, $u \in C^{1,\frac{1}{p+1} }(B_{1/2}(0))$, along free boundary points. More precisely, for any point $x_0 \in \partial \{u > \phi\} \cap B_{1/2}(0)$ there holds
$$
	[u]_{C^{1, \frac{1}{p+1}}(B_r(x_0))} \le C\cdot \left[\|u\|_{L^{\infty}(B_1(0))} + \left(\|\phi\|_{C^{1, 1}(B_1(0))}^{p+1}+\|f\|_{L^{\infty}(B_1(0))}\right)^{\frac{1}{p+1}}\right]	,
$$
for $0<r < \frac{1}{2}$ where $C>0$ is a universal constant.
\end{theorem}

In contrast with Theorems \ref{main3} and \ref{LGR}, we are able to prove the following non-degeneracy result. The proof makes use the same ideas as in \cite[Theorem 1.7]{DSVI} and  \cite[Theorem 1.7]{DSVII}. Thus, we will omit it here.

\begin{theorem}[{\bf Non-degeneracy property}]\label{ThmNon-Deg} Suppose that the assumptions of Theorem \ref{T1} are in force. Let $u$ be a bounded a viscosity solution to the obstacle problem \eqref{Eq1} with source term satisfying $\displaystyle \inf_{B_1(0)}f(x) \defeq \mathfrak{m}>0$. Given $x_0 \in \{u>\phi\} \cap \Omega^{\prime}$, then there holds
$$
  \displaystyle \sup_{\partial B_r(x_0)} (u(x)-\phi(x_0))\geq \mathfrak{c}\cdot r^{\frac{p+2}{p+1}}\quad \text{for all} \quad  0<r \ll 1.
$$
for a constant $\mathfrak{c} = \mathfrak{c}(\mathfrak{m},\|\mathfrak{a}\|_{L^{\infty}(\Omega)}, L_1, N, \lambda, \Lambda, p,q, \Omega^{\prime})>0$.
\end{theorem}

\section{Further applications}\label{Sec.FurtApplic}

We will connect our results to several known scenarios, which in some extent, we retrieve or extend.

\subsection{Doubly degenerate $p-$Laplacian in non-divergence form}

We would like to highlight that other interesting class of degenerate operators where our results work out is given by the double degenerate $p-$Laplacian type operator, in its non-divergence form, for $2<p\leq q< \infty$:
     $$
         \mathcal{G}_{p, q}(x, \xi, X) = \mathcal{H}_{p, q}(x, \xi)F_p(\xi, X)
     $$
where
   $$
     \mathcal{H}_{p, q}(x, \xi)\defeq |\xi|^{p-2}+ \mathfrak{a}(x)|\xi|^{q-2}
   $$
and
$$
     F_p(\xi, X) \defeq  \tr\left[\left(\textrm{Id}_N+(p-2)\frac{\xi\otimes \xi}{|\xi|^2}\right)X\right].
$$
is the Normalized $p-$Laplacian operator. At this point, notice that for an arbitrary $\nu \in \mathbb{S}^{N-1}$ we have
$$
\begin{array}{ccc}
  \left\langle \textrm{Id}_N + (p-2)\frac{Du \otimes Du}{|Du|^2}\nu, \nu\right\rangle & = & \displaystyle |\nu|^2 + (p-2)\frac{\langle \nu, Du \rangle^2 }{|Du|^2}\\
   & = & \displaystyle 1 + (p-2)\frac{\langle \nu, Du \rangle^2 }{|Du|^2}.
\end{array}
$$
Therefore, we conclude that $F_p$ satisfies assumption (A1) with
$$
   \lambda = \min\{p-1, 1\} \quad \mbox{and}\ \quad \Lambda = \max\{p-1, 1\}.
$$

In its simplest non-variational form, namely,
$$
   \Delta^N_p u \defeq |D u|^{2-p}\text{div}(|D u|^{p-2}D u) = \Delta u + (p-2)\left\langle D^2 u \frac{D u}{|D u|}, \frac{D u}{|D u|}\right\rangle,
$$
such an operator and similar normalized ones have attracted growing attention due to their interesting connections to stochastic zero-sum Tug-of-war games, whose pioneering result was first considered in Peres and Sheffield's seminal work \cite{PSSW09} via a game theoretical approach (see also \cite{BdaSR19} for a similar treatment related to free boundary problems).

Problems governed by the normalized $p-$Laplacian operator have attracted a huge deal of attention in the last years due to the wide wealth of geometric, pure PDE and probabilistic information related to solutions of such an operator. In this direction, we worth highlighting the series of fundamental works \cite{APR} and \cite{AR18} where the authors address $C_{\text{loc}}^{1, \alpha}$ regularity estimates to inhomogeneous problem
$$
   -\Delta^N_p u  = f\in L^{\infty}(B_1(0)) \,\,\text{and} \,\, -|Du|^{\gamma}\Delta^N_p u  = f\in L^{\infty}(B_1(0))\,\, \text{for} \,\gamma >-1.
$$

Finally, by combining the qualitative results from \cite{JLM01} (equivalence of notion of weak solutions) with quantitative ones (regularity estimates) from \cite[Theorem 1.1]{APR} and \cite[Theorem 1.1]{AR18} we can to obtain the following result:

\begin{theorem} Let $u$ be a bounded viscosity solution to the problem
\begin{equation}\label{EqNormp-Lapla}
   \mathcal{G}_{p, q}(x, Du, D^2 u)  =  f(x, u)  \textrm{ in } B_1(0)
\end{equation}
with $f \in C^0(B_1(0) \times \R)\cap L^\infty(B_1(0) \times \R)$. Then, $u \in C_{\text{loc}}^{1, \beta_p }(B_1(0))$. Moreover, the following estimate there holds
$$
         \displaystyle [u]_{C^{1,\beta_p}\left(B_{\frac{1}{2}}(0)\right)} \leq C\cdot \left[\|u\|_{L^{\infty}(B_1(0))} + 1 +\|f\|_{L^{\infty}(B_1(0) \times \R)}^{\frac{1}{p-1}}\right],
$$
where, $\beta_p \in (0, \alpha_p)\cap \left(0,  \frac{1}{p-1}\right]$, $0<r < \frac{1}{2}$, $C>0$ is a universal constant and $0< \alpha_p \leq 1$ is the sharp H\"{o}lder gradient exponent related to the homogeneous problem.
\end{theorem}

\begin{proof} The proof follows the same lines of Theorem \ref{main1}:
\end{proof}

We are able to obtain an explicit regularity exponent in the previous result in some particular scenarios.

\begin{corollary} Let $u$ be a bounded viscosity solution to \eqref{EqNormp-Lapla}. Suppose we are in some of scenarios considered in \cite{ATU17} and \cite[Section 4]{ATU18}. Then, $u \in C_{\text{loc}}^{1, \frac{1}{p-1}}(B_1(0))$. Moreover, there holds
	$$
	\displaystyle [u]_{C^{1, \frac{1}{p-1}}\left(B_{\frac{1}{2}}(0)\right)}\leq  C\cdot \left[\|u\|_{L^{\infty}(B_1(0))} + 1 +\|f\|_{L^{\infty}(B_1(0) \times \R)}^{\frac{1}{p-1}}\right].
	$$
\end{corollary}

\subsection{Connections with inhomogeneous $\infty-$Laplace equation}

Let us continue with the celebrated classical case: $\infty-$Laplace equation.

Next, we consider a normalized viscosity solution $(u, v)$  to the nonlinear system:

\begin{equation}\label{EqSyst2}
\left\{
\begin{array}{rcll}
  \left[1+\left(\frac{q-2}{p-2}\right)\mathfrak{a}(x)|Du|^{q-p}\right]\Delta_{\infty} u(x) & = & f(x, v) & \mbox{in} \,\, \Omega \\
 \left[|Dv|^{2}+\mathfrak{a}(x)|Dv|^{q-p+2}\right]\Delta v(x) & = & g(x, u) & \mbox{in} \,\,\,\Omega \\
 u(x) & = & h(x) & \mbox{on} \,\,\,\partial \Omega \\
 v(x) & = & h(x) & \mbox{on} \,\,\,\partial \Omega,\\
\end{array}
\right.
\end{equation}
where $2<p\leq q< \infty$, $f, g \in C^0(\Omega\times \R; \R)$, $h \in C^0(\partial \Omega)$ and \eqref{1.3} is in force, and
$$
 \Delta_{\infty} u(x) \defeq \left\langle D^2 u D u, D u\right\rangle
$$
states the nowadays well-known $\infty-$Laplacian operator (see, \cite{ACJ04} for a comprehensive survey on this subject and \cite{BdaSR19}, \cite{daSRS19I} and \cite{RTU15} for asymptotic analysis problems where such an operator arises).

Now, let us provide a few insights toward regularity aspects of the system \eqref{EqSyst2}. By exploring the scaling properties, we notice that
$$
\left|\text{First Equation in} \,\,\,\eqref{EqSyst2}\right| \lesssim \left[|Du|^{2}+\left(\frac{q-2}{p-2}\right)\mathfrak{a}(x)|Du|^{q-p+2}\right]\cdot |D^2 u|
$$
shares the same double degeneracy feature as
$$
\left[|Dv|^{2}+\mathfrak{a}(x)|Dv|^{q-p+2}\right]\cdot|\Delta v(x)|\lesssim 1
$$

Taking into account the regularity of the data and invoking estimates addressed in \cite[Corollary 2.]{L14} and \cite{LW08} we can conclude that $u$ is a locally Lipschitz continuous function with its corresponding norm uniformly controlled. On the other hand, from Corollary \ref{Cormain1} we conclude that $v \in C_{\text{loc}}^{1, \frac{1}{3}}(\Omega)$ with a uniform bound for its corresponding semi-norm (cf. \cite{RTU15}).

Next, let us present a simple example where the system \eqref{Eqp&qLaplac} takes place. Let $w$ be a viscosity solution to

\begin{equation}\label{Eqp&qLaplac}
\begin{array}{rcl}
  \mathcal{G}(Dw, D^2 w) & \defeq & |Dw|^{4-p}\left(\Delta_p w + \mathfrak{a}(x)\Delta_q w\right) \\
   & = & \mathcal{G}_1(Dw, D^2 w) + \mathcal{G}_2(Dw, D^2 w)\\
   & = & \text{A bounded source term} \quad \text{in} \quad B_1(0),
\end{array}
\end{equation}
where
$$
\left\{
\begin{array}{rcl}
  \mathcal{G}_1(Dw, D^2 w) & \defeq & \left[|Dw|^{2}+\mathfrak{a}(x)|Dw|^{q-p+2}\right]\Delta w(x) \\
   \mathcal{G}_2(Dw, D^2 w) & \defeq &  \left[(p-2)+\left(q-2\right)\mathfrak{a}(x)|Dw|^{q-p}\right]\Delta_{\infty} w(x).
\end{array}
\right.
$$
Observe that $\mathcal{G}_2$ represents the ``rough part'' of the operator because it does not satisfy the structural assumptions (A0)-(A2). Nevertheless, if it were possible to obtain a sort of uniform bound for such a term, we could obtain a similar conclusion as previously (cf. \cite[Proposition 4.6]{ART15} and \cite{RTU15}).

\begin{proposition} Let $w$ be a bounded viscosity solution to \eqref{Eqp&qLaplac}. Assume further $\mathcal{G}_2 \in L^{\infty}(B_1(0))$. Then, $w \in C_{\text{loc}}^{1, \frac{1}{3}}(B_1(0))$. Moreover, there holds
	$$
	\displaystyle [u]_{C^{1, \frac{1}{3}}\left(B_{\frac{1}{2}}(0)\right)}\leq  C\cdot \left[\|u\|_{L^{\infty}(B_1(0))} + 1 +\sqrt[3]{\|\mathcal{G}+\mathcal{G}_2\|_{L^{\infty}(B_1(0))}}\right].
	$$
\end{proposition}

A similar reasoning also works to viscosity solutions of

\begin{equation}\label{Eqp&qLaplac2}
\begin{array}{rcl}
  \mathcal{G}(Dw, D^2 w) & \defeq & |Du|^{\gamma}\left(\Delta_p w + \mathfrak{a}(x)\Delta_q w\right)\\
   & = & \mathcal{G}^{\ast}_1(Dw, D^2 w) + \mathcal{G}^{\ast}_2(Dw, D^2 w)\\
   & = & \text{A bounded source term} \quad \text{in} \quad B_1(0),
\end{array}
\end{equation}
where $\gamma > 2-p$, and for $\tau_p = p-2+\gamma$ and $\tau_q  = q-2+\gamma$ we have
$$
\left\{
\begin{array}{rcl}
  \mathcal{G}^{\ast}_1(Dw, D^2 w) & \defeq & \left[|Dw|^{\tau_p}+\mathfrak{a}(x)|Dw|^{\tau_q}\right]\Delta w(x) \\
   \mathcal{G}^{\ast}_2(Dw, D^2 w) & \defeq &  \left[(p-2)|Dw|^{\tau_p}+(q-2)\mathfrak{a}(x)|Dw|^{\tau_q}\right]\Delta_{\infty} w(x).
\end{array}
\right.
$$
One more time, a kind of uniform bound for the ``bad term'' in the previous equations implies a sort of fine regularity estimate for solutions (cf. \cite{AR18}).

\begin{proposition} Let $w$ be a bounded viscosity solution to \eqref{Eqp&qLaplac2}. Assume further $\mathcal{G}^{\ast}_2 \in L^{\infty}(B_1(0))$. Then, $w \in C_{\text{loc}}^{1, \frac{1}{\tau_p+1}}(B_1(0))$.
\end{proposition}

Particularly, when $\gamma = 0$ we recover, to some extent, the classical results regarding $C^{p^{\prime}}$-conjecture addressed in \cite[Theorem 1]{ATU17} and \cite[Theorem 2]{ATU18}.

\begin{corollary} Let $w$ be a bounded viscosity solution to \eqref{Eqp&qLaplac2}. Assume further $\mathcal{G}^{\ast}_2 \in L^{\infty}(B_1(0))$. Then, $w \in C_{\text{loc}}^{1, \frac{1}{p-1}}(B_1(0))$.
\end{corollary}

\subsection{Problems with $(p, q)-$growth in non-divergence form}

In conclusion, we would like to stress that our approach also allows us to consider solutions to nonlinear problems with $(p, q)-$growth as follows (cf. \cite{DeFM20}):
$$
-\mathcal{L}w(x) \defeq \Div\left(|Dw|^{p-2}Dw + a(x)|Dw|^{q-2}Dw\right) = f\in L^{\infty}(B_1(0))
$$
By way of exemplification, such solutions can be obtained as minimizers to
{\small{
$$
\displaystyle  w \mapsto \min_{u_0+W_0^{1,p}(\Omega)} \int_{\Omega} \left( \frac{1}{p}|Dw|^p + \frac{1}{q}\mathfrak{a}(x)|Dw|^q - fw\right)dx \quad (\text{with \eqref{1.3} in force})
$$}}
From the equivalence of weak solutions in \cite{FZ20}, and by formal computation such an operator can be re-written (in non-divergence form) as:
\begin{equation}\label{Eqp&qLaplac3}
  -\mathcal{L}w(x) \defeq \mathcal{G}_{1}^{\star}(Du, D^2 w) + \mathcal{G}_{2}^{\star}(Du, D^2 w)= f(x),
\end{equation}
where
$$
\left\{
\begin{array}{rcl}
  \mathcal{G}^{\star}_1(Dw, D^2 w) & \defeq & \left[|Dw|^{p-2}+\mathfrak{a}(x)|Dw|^{q-2}\right]\Delta w(x) \\
   \mathcal{G}^{\star}_2(Dw, D^2 w) & \defeq &  \left[\left(p-2\right)|Dw|^{p-2}+\left(q-2\right)\mathfrak{a}(x)|Dw|^{q-2}\right]\Delta^N_{\infty} w(x)\\
    &+& |Dw|^{q-2}Dw\cdot D \mathfrak{a}(x),
\end{array}
\right.
$$
where
$$
\Delta^N_{\infty} w(x) \defeq \left\langle D^2 w \frac{D w}{|D w|}, \frac{D w}{|D w|}\right\rangle
$$
is the normalized $\infty-$Laplacian operator (see, \cite{LW08II}).

Therefore, under some specific constraints we are able to obtain the following result (cf. \cite{ATU17} and \cite{ATU18}):
\begin{proposition} Let $w \in C^{0, 1}(B_1(0))$ be a bounded viscosity solution to \eqref{Eqp&qLaplac3}. Assume $\Delta_{\infty}^N w, D \mathfrak{a} \in L^{\infty}(B_1(0))$. Then, $w \in C_{\text{loc}}^{1, \frac{1}{p-1}}(B_1(0))$. Moreover, there holds
	$$
	\displaystyle [w]_{C^{1, \frac{1}{p-1}}\left(B_{\frac{1}{2}}(0)\right)}\leq  C\cdot \left[\|w\|_{L^{\infty}(B_1(0))} + 1 +\|f+\mathcal{G}^{\star}_2\|_{L^{\infty}(B_1(0))}^{\frac{1}{p-1}}\right].
	$$
\end{proposition}

\section{Some explicit examples and further comments}\label{Examples}

\subsection{Some examples}

We will present some examples where we obtain optimal estimates. Precisely, our examples are in such a way that $\alpha_{\mathrm{F}}=1$. Therefore, we are allowed to choose $\beta = \frac{1}{p+1}$ in the Theorem \ref{main1}.

\begin{example}[{\bf Sharpness of Theorem \ref{main1}}]
It is important to stress that one cannot expect solutions of \eqref{1.1} to be in $C^{1, 1}$ even when the data is smooth. Indeed, the radially symmetric function $v: B_1(0)\to \R$ given by
$$
 \,v(x)=|x|^{\,\frac{p+2}{p+1}}
$$
fulfills
$$
   \left[|D v|^{p}+\mathfrak{a}(x)|D v|^{q}\right]\Delta v  =  f(x)  \text{ in }  B_1(0)
$$
in the viscosity sense, where $\mathfrak{a} \in C^0(B_1(0), [0, \infty))$ and
$$
\begin{array}{rcl}
  f(x) & = & \left[\left(\frac{p+2}{p+1}\right)^{p+1}+\left(\frac{p+2}{p+1}\right)^{q+1}\right]\left[\frac{1}{p+1}+(N-1)\right]\left(1+\mathfrak{a}(x)|x|^{\frac{q-p}{p+1}}\right) \\
   & \in & L^{\infty}(B_1(0))\cap C^0(B_1(0)).
\end{array}
$$
Note that $v\in C_{\text{loc}}^{1,\frac{1}{p+1}}(B_1(0))$, however $v\notin C_{\text{loc}}^{1,\frac{1}{p+1}+\varepsilon}(B_1(0))$ for any $\varepsilon \in (0, 1)$.
\end{example}

\begin{example}[{\bf Operators with small ellipticity aperture}] Our results hold for operators with small ellipticity aperture: For such a class, local $C^{2,\alpha}$ \textit{a priori} estimates for solutions of fully nonlinear equations hold under the assumption that the ellipticity constants $0<\lambda\le \Lambda<\infty$ do not deviate much, in the sense that $\mathfrak{e} \defeq 1 -\frac{\lambda}{\Lambda}$ is small enough (see, \cite[Chapter 5]{daS15}). Particularly, such operators cover Isaac's type equations, which appear in stochastic control and in the theory of differential games:
      $$
      \displaystyle F(x, D^2 u) \defeq \sup_{\hat{\beta} \in \mathcal{B}} \inf_{\hat{\alpha} \in \mathcal{A}}\left(\mathcal{L}^{\hat{\alpha}\hat{\beta}} u(x)\right) \quad \left(\text{resp.}\,\,\,\inf_{\hat{\beta} \in \mathcal{B}} \sup_{\hat{\alpha} \in \mathcal{A}} \left(\mathcal{L}^{\hat{\alpha}\hat{\beta}} u(x)\right)\right),
      $$
      where
      $$
         \displaystyle \mathcal{L}^{\hat{\alpha}\hat{\beta}} u(x) = \sum_{i, j=1}^{N} a_{ij}^{\hat{\alpha}\hat{\beta}}(x)\partial_{ij} u(x)
      $$
      is a family of uniformly elliptic operators with H\"{o}lder continuous coefficients and ellipticity constants $\lambda$ and $\Lambda$ such that $0<\left|1 -\frac{\lambda}{\Lambda}\right| \ll 1$.

\end{example}

\begin{example}[{\bf Concave/convex operators}]\label{Ex3} Our results hold for Pucci's extremal operators
$$
   F(D^2u) \defeq \mathcal{M}_{\lambda, \Lambda}^{\pm}(D^2 u)
$$
More generally, we can consider Belmann's type equations, which appear in stochastic control:
\[
\displaystyle F(x, D^2 u) \defeq  \inf_{\hat{\alpha} \in \mathcal{A}}\left(\mathcal{L}^{\hat{\alpha}} u(x)\right) \quad \left(\text{resp.}\,\,\,\sup_{\hat{\alpha} \in \mathcal{B}}  \left(\mathcal{L}^{\hat{\alpha}} u(x)\right)\right),
\]
where
      $$
         \displaystyle \mathcal{L}^{\hat{\alpha}} u(x) = \sum_{i, j=1}^{N} a_{ij}^{\hat{\alpha}}(x)\partial_{ij} u(x)
      $$
is a family of uniformly elliptic operators with H\"{o}lder continuous coefficients, and ellipticity constants $0<\lambda \le\Lambda<\infty$.

In conclusion, we can consider concave/convex operators, where $C_{\text{loc}}^{2, \alpha}$ estimates are available for homogeneous problems, according to Evans-Krylov-Trudinger's regularity theory \cite{Ev82}, \cite{Kry82}, \cite{Kry83}, \cite{Tru83} and \cite{Tru84}.

Therefore, we get, in such a scenario, $C_{\text{loc}}^{1, \frac{1}{p+1}}$ estimates to \eqref{1.1}.
\end{example}

\begin{example}[{\bf Flat solutions for non-convex operators}] Recently, in \cite[Theorem 1]{DD19} was established, local Schauder type estimates for non-convex fully nonlinear operators ($C_{\text{loc}}^{2, \alpha}$ \textit{a priori} estimates) for flat solutions, i.e., solutions whose oscillation is very small, provided $F \in C^{1, \tau}(\text{Sym}(N))$ and has continuous coefficients. Therefore, such a family of solutions and operators, we obtain $C_{\text{loc}}^{1, \frac{1}{p+1}}$ estimates to \eqref{1.1}.
\end{example}

\begin{example}[{\bf Estimates for solutions to ``almost linear'' equations}]
In the recent work \cite{BW}, Bhattacharya and Warren establish $C_{\text{loc}}^{2,\alpha}$ estimates for solutions to almost linear elliptic equations, i.e., fully nonlinear models, which are close to linear ones, provided $F$ is $C^1$-close to a linear operator. Therefore, such a family of operators, we obtain $C_{\text{loc}}^{1, \frac{1}{p+1}}$ estimates to \eqref{1.1}, which is the optimal regularity.
\end{example}

\begin{example}[{\bf Estimates for a class of non-convex operators}]
An interesting application of our results when $F$ belongs to a class of non-convex, fully nonlinear equation approached by Cabr\'{e} and Caffarelli in \cite{CC03}. In such a scenario, our results give $C_{\text{loc}}^{1,\frac{1}{p+1}}$ regularity estimates, which is the optimal to \eqref{1.1}.
\end{example}

\subsection{Recession operator and its regularity estimates}

Regarding the convexity or concavity assumptions on the nonlinearity $F$ in the Example \ref{Ex3}, we can actually relax them. For this purpose, the key ingredient is an available $C_{\text{loc}}^{1, \alpha}$ (for every $\alpha \in (0, 1)$) regularity theory to
$$
  F(D^2u) = 0 \quad \text{in} \quad \Omega.
$$

In that respect, Silvestre-Teixeira in \cite[Theorems 1.1 and 1.4]{ST15} addressed local $C^{1, \alpha}$ regularity estimates to problems with no convex/concave structure. In their approach, the novelty with respect to the former results is the concept of \textit{recession} function, a sort of tangent profile for $F$ at ``infinity'':
$$
  \displaystyle F^{\ast}(X) \defeq  \lim_{\tau \to 0+ } \tau F\left(\frac{1}{\tau}X\right)
$$
In this direction, the authors relaxed the hypothesis of $C_{\text{loc}}^{1,1}$ \textit{a priori} estimates for solutions of the equations without dependence on $x$, by the hypothesis that $F$ is assumed to be ``convex or concave only at the ends of $\textit{Sym}(N)$'', i.e., when $\|D^2 u\| \approx \infty$.
Precisely, they proved that if solutions to the homogeneous equation
$$
F^{\ast}(D^2 u) = 0  \quad \textrm{in} \quad B_1(0)
$$
has $C_{\text{loc}}^{1, \alpha_0}$ \textit{a priori} estimates (for some $\alpha_0 \in (0, 1]$), then solutions to
$$
 	F(D^2 u) = 0  \quad \textrm{in} \quad B_1(0) \quad (\text{resp.}  = f \in L^{\infty})
$$
are $C^{1, \hat{\alpha}}_{\text{loc}}(B_1(0))$ for any $\hat{\alpha} <\min\{1, \alpha_0\}$. We recommend the reader to \cite[Section 1]{daSilRic} and \cite[Section 1]{PT16} for a complete state of the art on this subject.

Therefore, if the recession profile associated to $F$, i.e. $F^{\ast}$, enjoys $C_{\text{loc}}^{1,1}$ \textit{a priori} estimates, then a ``good regularity theory'' is available to solutions of $F(D^2u) = 0$. For this reason, we are able to prove our results to operators under either relaxed or no convexity assumptions on $F$.

As commented above, our results hold for those operators whose recession profile enjoys a suitable \textit{a priori} estimate  (see, \cite[Theorems 1.1 and 1.4]{ST15}).

\begin{example}
For didactic reasons, we will present some operators and their recession counterpart.
For that end, consider $0 < \sigma_1, \ldots , \sigma_N< \infty.$ We have the following examples (see, \cite{daSilRic} and \cite{PT16}):
\begin{enumerate}

  \item[(E1)]({\bf $m$-momentum type operators})
Let $m$ be a fixed odd number. Then, the \textit{$m$-momentum type operator} given by
  $$
  \begin{array}{ccc}
    F_m(D^2 u) & = & F_m(e_1(D^2 u), \cdots, e_n(D^2 u)) \\
     & \defeq & \displaystyle \sum_{j=1}^{N} \sqrt[m]{\sigma_j^m+e_j(D^2 u)^m}-\sum_{j=1}^{N} \sigma_j
  \end{array}
    $$
defines a uniformly elliptic operator which is neither concave nor convex. Moreover,
$$
  \displaystyle F_m^{\ast}(X) = \lim_{\tau \to 0+} \tau F_m\left(\frac{1}{\tau} X\right) = \sum_{j=1}^{N} e_j(X)
$$
is the Laplacian operator.

\item[(E2)]({\bf Perturbation of Pucci's operators})
Let us consider
  $$
  \begin{array}{rcl}
    F(D^2 u) & = &  F(e_1(D^2 u), \cdots, e_N(D^2 u)) \\
     & \defeq & \displaystyle \sum_{j=1}^{N} \left[h(\sigma_j)e_j(D^2 u) + g(e_j(D^2 u))\right],
  \end{array}
  $$
  where $h:(0, \infty) \to (0, \infty)$ is a continuous function and $g: \R \to \R$ is any Lipschitz function such that $g(0) = 0$. Notice that $F$ is a uniformly elliptic operator.
  Moreover,
  $$
  \displaystyle F^{\ast}(X) = \lim_{\tau \to 0+} \tau F \left(\frac{1}{\tau} X\right) = \sum_{j=1}^{N} h(\sigma_j)e_j(X),
  $$
which is, up to a change of coordinates, the Laplacian operator.

\item[(E3)]({\bf Perturbation of the Special Lagrangian equation})
Given $h:[0, \infty) \to \R$ a continuous function, the \textit{``perturbation'' of the Special Lagrangian equation}
$$
  \begin{array}{rcl}
    F(D^2 u) & = &  F(e_1(D^2 u), \cdots, e_N(D^2 u)) \\
     & \defeq & \displaystyle \sum_{j=1}^{N} \left[h(\sigma_j)e_j(D^2 u) + \arctan(e_j(D^2 u))\right],
  \end{array}
$$
defines a uniformly elliptic operator which is neither concave nor convex. Furthermore,
  $$
  \displaystyle F^{\ast}(X) = \lim_{\tau \to 0+} \tau F \left(\frac{1}{\tau} X\right) = \sum_{j=1}^{N} h(\sigma_j)e_j(X),
  $$
which is, once again, a change of the coordinates of the Laplace operator.
\end{enumerate}
\end{example}

\begin{example}
Consider $F:\textit{Sym}(N) \to \R$ a $C^1$ uniformly elliptic operator. Then the \textit{recession profile} $F^{\ast}$ should be understood as the ``limiting equation'' for the natural scaling on $F$. By way of clarification, for a number of operators, it is possible to check the existence of the limit
$$
  \mathfrak{A}_{ij} \defeq \lim_{\|X\|\to \infty} \frac{\partial F}{X_{ij}}(X),
$$
In such a situation we obtain $F^{\ast}(X) = \tr(\mathfrak{A}_{ij}X)$. An interesting illustrative example is the class of Hessian type operators:
$$
\displaystyle  F_m(e_1(D^2 u), \cdots, e_N(D^2 u)) \defeq \sum_{j=1}^{N} \sqrt[m]{1+e_j(D^2 u)^m}-N,
$$
where $m \in \mathbb{N}$ is an odd number. In this scenario,
$$
  \displaystyle F^{\ast}(X) = \sum_{j=1}^{N} e_j(X) \quad (\text{the Laplacian operator}).
$$
\end{example}

\noindent{\bf Acknowledgements.} J.V. da Silva has been partially supported by Coordena\c{c}\~{a}o de Aperfei\c{c}oamento de Pessoal de N\'{i}vel Superior (PNPD/Capes-UnB-Brazil) Grant No. 88887.357992/2019-00, Capes-Brazil Finance Code 001 and CNPq-Brazil under Grant No. 310303/2019-2. J.V. da Silva thanks to IMECC from Universidade Estadual de Campinas - UNICAMP for its warm hospitality and for providing a pleasant working environment during his research visit, where part of this manuscript was written. Gleydson C. Ricarte thanks to Analysis research group of UFC for fostering a pleasant and productive scientific atmosphere. G.C. Ricarte has been partially funded by CNPq-Brazil  under Grant No.303078/2018-9.

We would like to thank the anonymous Referee for insightful comments and suggestions which benefited a lot the final outcome of this manuscript.

\end{document}